\documentclass[reqno]{amsart}
\usepackage{amsmath, amsthm, amssymb, amstext}

\usepackage[colorlinks=true]{hyperref}
\usepackage{color}
\usepackage{enumitem}
\newtheorem{theorem}{Theorem}
\newtheorem{remark}[theorem]{Remark}
\newtheorem{lemma}[theorem]{Lemma}
\newtheorem{proposition}[theorem]{Proposition}
\newtheorem{corollary}[theorem]{Corollary}
\newtheorem{definition}[theorem]{Definition}
\newtheorem{example}[theorem]{Example}

\DeclareMathOperator*{\im}{im}              %
\DeclareMathOperator*{\rank}{rank}              %
\DeclareMathOperator*{\dom}{dom}              %
\DeclareMathOperator*{\divergenz}{div}              %
\DeclareMathOperator*{\ints}{int}         %
\DeclareMathOperator*{\essinf}{ess  inf}         %
\DeclareMathOperator*{\id}{id}         %

\def\Z{\mathbb{Z}}

\def\R{\mathbb{R}}
\def\C{\mathbb{C}}
\def\RN{\mathbb{R}^N}

\def\W1p{W^{1,p}(\Omega)}
\def\W1p0{W^{1,p}_0(\Omega)}
\def\Wx1p{W^{1,p(\cd)}(\Omega)}

\newcommand{\Lp}[1]{L^{#1}(\Omega)}

\def\lan{\langle}
\def\ran{\rangle}
\def\eps{\varepsilon}
\def\ph{\varphi}
\def\Om{\Omega}

\def\into{\int_{\Omega}}
\def\weak{\rightharpoonup}

\def\Linf{L^{\infty}(\Omega)}
\def\C1{\ints (C^1(\overline{\Om})_+)}

\def\close{\overline{\Omega}}

\def\interior{\ints \left(C^1_0(\overline{\Omega})_+\right)}

\numberwithin{theorem}{section}
\numberwithin{equation}{section}

\def\cprime{$'$}

\renewcommand{\l}{\left}
\renewcommand{\r}{\right}

\title[Nonlinear nonhomogeneous Dirichlet equations]{Nonlinear nonhomogeneous Dirichlet equations with a superlinear reaction}

\author[N. S. Papageorgiou]{Nikolaos S. Papageorgiou}
\address[N. S. Papageorgiou]{National Technical University, Department of Mathematics, Zografou Campus, Athens 15780, Greece}
\email{npapg@math.ntua.gr}

\author[P. Winkert]{Patrick Winkert}
\address[P. Winkert]{Technische Universit\"{a}t Berlin, Institut f\"{u}r Mathematik, Stra\ss e des 17. Juni 136, 10623 Berlin, Germany}
\email{winkert@math.tu-berlin.de}


\begin{document}

\begin{abstract}
    We consider a nonlinear elliptic Dirichlet equation driven by a nonlinear nonhomogeneous differential operator involving a Carath\'{e}odory reaction which is $(p-1)$-superlinear but does not satisfy the Ambrosetti-Rabinowitz condition. First we prove a three-solutions-theorem extending an earlier classical result of Wang (Ann. Inst. H. Poincar\'e Anal. Non Lin\'eaire 8 (1991), no. 1, 43--57). Subsequently, by imposing additional conditions on the reaction $f(x,\cdot)$, we produce two more nontrivial constant sign solutions and a nodal solution for a total of five nontrivial solutions. In the special case of $(p,2)$-equations we prove the existence of a second nodal solution for a total of six nontrivial solutions given with complete sign information. Finally, we study a nonlinear eigenvalue problem and we show that the problem has at least two nontrivial positive solutions for all parameters $\lambda>0$ sufficiently small where one solution vanishes in the Sobolev norm as $\lambda \to 0^+$ and the other one blows up (again in the Sobolev norm) as $\lambda \to 0^+$.
\end{abstract}

\maketitle

\section{Introduction}
Let $\Omega \subseteq \R^N$ be a bounded domain with a $C^2$-boundary $\partial \Omega$ and let $1<p<\infty$. In this paper, we study the following nonlinear nonhomogeneous Dirichlet problem
\begin{equation}\label{problem}
    \begin{aligned}
      -\divergenz a(\nabla u(x)) & = f(x,u(x)) \quad && \text{in } \Omega,\\
       u & = 0  &&\text{on } \partial \Omega,
    \end{aligned}
\end{equation}
where $a: \R^N \to \R^N$ is a continuous, strictly monotone map which is $C^1$ on $\R^N \setminus \{0\}$. The precise conditions on $a(\cdot)$ are given in hypotheses H(a) below. These conditions are general enough to incorporate some differential operators of interest in our framework like the $p$-Laplacian ($1<p<\infty$), the $(p,q)$-Laplacian ($1<q<p<\infty$) and the generalized $p$-mean curvature differential operator ($1<p<\infty$). The reaction $f: \Omega \times \R \to \R$ is assumed to be a Carath\'{e}odory function (i.e., $x \mapsto f(x,s)$ is measurable for all $s\in \R$ and $s \mapsto f(x,s)$ is continuous for a.a. $x \in \Omega$) which exhibits $(p-1)$-superlinear growth near $\pm \infty$ but without satisfying the usual in such cases Ambrosetti-Rabinowitz condition. Our goal is to prove multiplicity theorems for such problems. For equations driven by the $p$-Laplacian, such multiplicity results were proved by Bartsch-Liu \cite{2004-Bartsch-Liu}, Bartsch-Liu-Weth \cite{2005-Bartsch-Liu-Weth}, Liu \cite{2010-Liu}, 
Papageorgiou-Rocha-Staicu \cite{2008-Papageorgiou-Rocha-Staicu} and Sun \cite{2010-Sun}.

Recall that, if $f: \Omega \times \R \to \R$ is a Carath\'{e}odory function and $F(x,s)=\int^s_0 f(x,t)dt$, we say that $f(x,\cdot)$ satisfies the Ambrosetti-Rabinowitz condition if there exist $\mu>p$ and $M>0$ such that
\begin{align}
    & 0<\mu F(x,s) \leq f(x,s)s \quad \text{for a.a. }x \in \Omega \text{ and for all } |s| \geq M, \label{AR1}\\
    & 0 < \essinf_{\Omega} F(\cdot, \pm M), \label{AR2}
\end{align}
(see Ambrosetti-Rabinowitz \cite{1973-Ambrosetti-Rabinowitz}). Integrating \eqref{AR1} and using \eqref{AR2}, we obtain the following growth conditions for the primitive $F(x,\cdot)$
\begin{align}\label{AR3}
    \tilde{\eta} |s|^\mu \leq F(x,s) \quad \text{for a.a. }x \in \Omega \text{, for all } |s| \geq M \text{, and some }\tilde{\eta}>0.
\end{align}
Thanks to \eqref{AR3} we have the much weaker condition
\begin{align}\label{AR4}
    \lim_{s \to \pm \infty} \frac{F(x,s)}{|s|^\mu}=+\infty \quad \text{uniformly for a.a. } x \in \Omega.
\end{align}
This means that the primitive $F(x,\cdot)$ is $(p-1)$-superlinear for a.a. $x\in \Omega$. In this paper we employ \eqref{AR4} combined with another asymptotic condition (see H(f)$_1$(iii)), which together are weaker than the Ambrosetti-Rabinowitz condition (see \eqref{AR1}, \eqref{AR2}) and fit in our analysis superlinear reactions with slower growth near $\pm \infty$.

The Ambrosetti-Rabinowitz condition, although very convenient in checking the Palais-Smale condition for the energy functional, is rather restrictive as revealed in the discussion above. So there have been efforts to relax it. For an overview of the relevant literature we refer to the recent works of Liu \cite{2010-Liu}, Li-Yang \cite{2010-Li-Yang}, and Miyagaki-Souto \cite{2008-Miyagaki-Souto}.

Our tools come from critical point theory and from Morse theory (critical groups) and involve also truncation and comparison techniques. In the next section, for the reader's convenience, we review the main definitions and facts which will employ in this work. We also introduce the hypotheses on the map $a(\cdot)$ and establish some useful consequences of these conditions.

\section{Preliminaries and hypotheses}\label{section_hypotheses}

Let $X$ be a Banach space and $X^*$ its topological dual while $\lan \cdot,\cdot\ran$ denotes the duality brackets to the pair $(X^*,X)$. We have the following definition.
\begin{definition}
  The functional $\ph \in C^1(X)$ fulfills the Cerami condition (the $C$-condition for short) if the following holds:
  every sequence $(u_n)_{n \geq 1} \subseteq X$ such that
      $(\ph(u_n))_{n \geq 1}$ is bounded in $\R$ and $(1+\|u_n\|_X)\ph'(u_n) \to 0$ in $X^*$ as $n \to \infty$, admits a strongly convergent subsequence.
\end{definition}

This compactness type condition on $\ph$ is more general than the well-known Palais-Smale condition which we encounter more often in the literature. Nevertheless, the $C$-condition suffices to have a deformation theorem from which one derives the minimax theory of certain critical values of $\ph$. One result of this theory is the so-called mountain pass theorem.

\begin{theorem}\label{theorem_mountain_pass}
    Let $\varphi\in C^1(X)$ be a functional satisfying the $C$-condition and let $u_1,u_2 \in X, \|u_2-u_1\|> \rho>0$,
    \begin{align*}
	\max \{\varphi(u_1),\varphi(u_2)\}<\inf \{\varphi(u): \|u-u_1\|_X=\rho\}=:\eta_\rho
    \end{align*}
    and $c=\inf_{\gamma \in \Gamma} \max_{0 \leq t \leq 1} \varphi(\gamma(t))$ with $\Gamma=\{\gamma \in C\l([0,1],X\r): \gamma(0)=u_1, \gamma(1)=u_2\}$. 
    Then $c \geq \eta_\rho$ with $c$ being a critical value of $\varphi$.
\end{theorem}

In the analysis of problem \eqref{problem} in addition to the Sobolev space $\W1p0$ we will also use the ordered Banach space
\begin{align*}
    C^1_0(\overline{\Omega})= \left\{u \in C^1(\overline{\Omega}): u\big|_{\partial \Omega}=0 \right\}
\end{align*}
and its positive cone
\begin{align*}
    C^1_0(\overline{\Omega})_+=\left\{u \in C^1_0(\overline{\Omega}): u(x) \geq 0 \text{ for all } x \in \overline{\Omega}\right\}.
\end{align*}
This cone has a nonempty interior given by
\begin{align*}
    \ints \left(C^1_0(\overline{\Omega})_+\right)=\left\{u \in C^1_0(\overline{\Omega})_+: u(x)>0 \text{ for all } x \in \Omega \text{ and } \frac{\partial u}{\partial n}(x)<0 \text{ for all } x \in \partial \Omega \right\},
\end{align*}
where $n(\cdot)$ stands for the outward unit normal on $\partial \Omega$.

Let $\vartheta \in C^1(0,+\infty)$ be a function satisfying
\begin{align}\label{hypotheses_theta}
    0 < \hat{c}\leq \frac{t \vartheta'(t)}{\vartheta(t)} \leq c_0 \quad \text{ and } \quad c_1 t^{p-1} \leq \vartheta(t) \leq c_2 (1+t^{p-1})
\end{align}
for all $t>0$ and with some constants $\hat{c}, c_0, c_1,c_2>0$.

Then the hypotheses on $a(\cdot)$ are the following.
\begin{enumerate}[leftmargin=1.2cm]
    \item[H(a):]
	$a(\xi)=a_0(\|\xi\|)\xi$ for all $\xi \in \R^N$ with $a_0(t)>0$ for all $t>0$ and\\
	\begin{enumerate}
	    \item[(i)]
		$a_0 \in C^1(0,\infty), t\mapsto ta_0(t)$ is strictly increasing, $\lim_{t \to 0^+} t a_0(t)=0$, and $\displaystyle \lim_{t \to 0^+} \frac{t a_0'(t)}{a_0(t)}>-1$;
	    \item[(ii)]
		$\displaystyle\|\nabla a(\xi)\| \leq c_3 \frac{\vartheta\l(\|\xi\|\r)}{\|\xi\|}$ for all $\xi \in \R^N \setminus \{0\}$ and some $c_3>0$;\\
	    \item[(iii)]
		$\displaystyle\l(\nabla a(\xi) y,y\r)_{\R^N} \geq \frac{\vartheta \l(\|\xi\|\r)}{\|\xi\|} \|y\|^2$ for all $\xi \in \R^N \setminus\{0\}$ and all $y \in \R^N$.\\
	\end{enumerate}
\end{enumerate}

\begin{remark}
    Owing to hypothesis H(a)(i) it follows that $a \in C^1\l(\R^N\setminus\{0\},\R^N\r)\cap C\l(\R^N,\R^N\r)$ and hence, hypotheses H(a)(ii), (iii) make sense. Let $G_0(t)=\int^t_0 s a_0(s) ds$ and let $G(\xi)=G_0(\|\xi\|)$ for all $\xi \in \R^N$. Then
    \begin{align*}
	\nabla G(\xi)=G'_0(\|\xi\|)\frac{\xi}{\|\xi\|}=a_0(\|\xi\|)\xi=a(\xi) \quad \text{ for all } \xi \in \R^N\setminus\{0\},
    \end{align*}
    which means that $G(\cdot)$ is the primitive of $a(\cdot)$. Obviously, $G(\cdot)$ is convex and since $G(0)=0$ we have the estimate
    \begin{align}\label{estimate}
	G(\xi) \leq (a(\xi),\xi)_{\R^N} \quad \text{for all }\xi \in \R^N.
    \end{align}
\end{remark}

These hypotheses have some interesting consequences on the map $a(\cdot)$.

\begin{lemma}\label{lemma_properties}
    Let the hypotheses H(a) be satisfied. Then there hold
    \begin{enumerate}[leftmargin=0.8cm]
	\item[(a)]
	    $\xi \to a(\xi)$ is maximal monotone and strictly monotone;
	\item[(b)]
	    $\|a(\xi)\| \leq c_4 (1+\|\xi\|^{p-1})$ for all $\xi \in \R^N$ and some $c_4>0$;
	\item[(c)]
	    $(a(\xi),\xi)_{\R^N} \geq \frac{c_1}{p-1}\|\xi\|^p$ for all $\xi \in \R^N$.
    \end{enumerate}
\end{lemma}

Taking into account Lemma \ref{lemma_properties} combined with \eqref{estimate} we infer the following growth estimates for the primitive $G(\cdot)$.

\begin{corollary}\label{corollary_upper_lower_estimates}
    If hypotheses H(a) hold, then
    \begin{align*}
	\frac{c_1}{p(p-1)}\|\xi\|^p \leq G(\xi) \leq c_5\l(1+\|\xi\|^p\r) \quad \text{for all } \xi \in \R^N \text{ and some } c_5>0.
    \end{align*}
\end{corollary}

\begin{example}\label{example_a}
    The following maps satisfy hypotheses H(a):
    \begin{enumerate}[leftmargin=0.8cm]
	\item[(a)]
	    $a(\xi)=\|\xi\|^{p-2}\xi$ with $1<p<\infty$.\\
	    This map corresponds to the $p$-Laplacian defined by
	    \begin{align*}
		\Delta_p u=\divergenz (\|\nabla u\|^{p-2} \nabla u) \quad \text{for all } u \in W^{1,p}_0(\Omega).
	    \end{align*}
	\item[(b)]
	    $a(\xi)=\|\xi\|^{p-2}\xi+\|\xi\|^{q-2}\xi$ with $1<q<p<\infty$.\\
	    This map corresponds to the $(p,q)$-differential operator defined by 
	    \begin{align*}
		\Delta_p u+ \Delta_q u \quad \text{for all }u \in W^{1,p}_0(\Omega).
	    \end{align*}
	    Note that this operator arises in problems of mathematical physics such as quantum physics (see Benci-D'Avenia-Fortunato-Pisani \cite{2000-Benci-D'Avenia-Fortunato-Pisani}) and in plasma physics and biophysics (see Cherfils-Il{\cprime}yasov \cite{2005-Cherfils-Il'yasov}).
	\item[(c)]
	    $a(\xi)=\l(1+\|\xi\|^2\r)^{\frac{p-2}{2}}\xi$ with $1<p<\infty$.\\
	    This operator represents the generalized $p$-mean curvature differential operator defined by
	    \begin{align*}
		\divergenz \left [(1+\|\nabla u\|^2)^{\frac{p-2}{2}} \nabla u \right] \quad \text{ for all } u \in W^{1,p}_0(\Omega).
	    \end{align*}
	\item[(d)]
	    $a(\xi)=\|\xi\|^{p-2}\xi \left( 1 +\frac{1}{1+\|\xi\|}\right)$ with $1 < p<\infty$. 
    \end{enumerate}
\end{example}

Now, let $f_0 : \Omega \times \R \to \R$ be a Carath\'{e}odory function with subcritical growth in $s \in \R$, that is
\begin{align*}
    |f_0(x,s)| \leq a(x)\left(1+|s|^{r-1}\right) \quad \text{ for a.a. } x \in \Omega \text{, and all } s \in \R,
\end{align*}
with $a \in L^\infty(\Omega)$, and $1<r<p^*$, where $p^*$ is the critical exponent of $p$ given by
\begin{align*}
        p^*=
        \begin{cases}
            \frac{Np}{N-p} \quad & \text{ if } p <N, \\
            +\infty & \text{ if } p \geq N.
        \end{cases} \qquad
\end{align*}
Let $F_0(x,s)=\int^s_0f_0(x,t)dt$ and let $\varphi_0: W^{1,p}_0(\Omega) \to \R$ be the $C^1$-functional defined by
\begin{align*}
    \varphi_0(u)=\int_\Omega G(\nabla u(x)) dx- \int_\Omega F_0(x,u(x)) dx.
\end{align*}
The following result, originally due to Brezis-Nirenberg \cite{1993-Brezis-Nirenberg}, can be found in Gasi{\'n}ski-Papageorgiou \cite{2012-Gasinski-Papageorgiou}. We also refer to earlier results in this direction in Garc{\'{\i}}a Azorero-Manfredi-Peral-Alonso \cite{2000-Garcia-Azorero-Peral-Alonso-Manfredi} and more recently, in Motreanu-Papageorgiou \cite{2011-Motreanu-Papageorgiou} and Winkert \cite{2010-Winkert}.

\begin{proposition}\label{proposition_local_minimizers}
  Let the assumptions in H(a) be satisfied. If $u_0 \in W^{1,p}_0(\Omega)$ is a local $C^1_0(\overline{\Omega})$-minimizer of $\varphi_0$, i.e., there exists $\rho_0>0$ such that
  \begin{align*}
      \varphi_0(u_0) \leq \varphi_0(u_0+h) \quad \text{for all } h \in C^1_0(\overline{\Omega}) \text{ with } \|h\|_{C^1_0(\overline{\Omega})} \leq \rho_0,
  \end{align*}
  then $u_0 \in C^{1,\beta}_0(\overline{\Omega})$ for some $\beta \in (0,1)$ and $u_0$ is also a local $W^{1,p}_0(\Omega)$-minimizer of $\varphi_0$, i.e., there exists $\rho_1>0$ such that
  \begin{align*}
      \varphi_0(u_0) \leq \varphi_0(u_0+h) \quad \text{for all } h \in W^{1,p}_0(\Omega) \text{ with } \|h\|_{W^{1,p}_0(\Omega)} \leq \rho_1.
  \end{align*}
\end{proposition}

Now, let $\frac{1}{p}+\frac{1}{p'}=1$ and let $A : W^{1,p}_0(\Omega) \to \l(\W1p0\r)^*=W^{-1,p'}(\Omega)$ be the nonlinear map defined by
\begin{align}\label{def_operator}
    \left \lan A(u),v \right \ran= \into (a(\nabla u),\nabla v)_{\R^N} dx \quad \text{for all } u,v \in W^{1,p}_0(\Omega).
\end{align}

Thanks to the results of Gasi{\'n}ski-Papageorgiou \cite{2008-Gasinski-Papageorgiou}) the operator $A$ has the following properties.

\begin{proposition}\label{proposition_basic_properties}
    Under hypotheses H(a) the operator $A: W^{1,p}_0(\Omega) \to W^{-1,p'}(\Omega)$ defined by (\ref{def_operator}) is bounded, continuous, monotone (hence maximal monotone) and of type (S)$_+$, i.e., if $u_n \weak u$ in $W^{1,p}_0(\Omega)$ and  $\limsup_{n \to \infty} \left \lan A(u_n),u_n-u \right \ran \leq 0$, then $u_n \to u$ in $W^{1,p}_0(\Omega)$.
\end{proposition}

Given $1<r<\infty$, the $r$-Laplacian $\Delta_r$  is a special case of $A$ which is defined by
\begin{align*}
    \l \lan \Delta_r(u),v \r \ran=\into \|\nabla u\|^{r-2} (\nabla u,\nabla v)_{\R^N} dx \quad \text{for all } u,v \in W^{1,r}_0(\Omega).
\end{align*}
If $r=2$, then $\Delta_r=\Delta$ becomes the well-known Laplace operator.

Let us recall some basic facts about the spectrum of the $r$-Laplacian with Dirichlet boundary condition.
Consider the nonlinear eigenvalue problem
\begin{equation}\label{eigenvalue_problem}
    \begin{aligned}
      -\Delta_r u(x) & = \hat{\lambda} |u(x)|^{r-2} u(x) \quad && \text{in } \Omega,\\
       u & = 0  &&\text{on } \partial \Omega,
    \end{aligned}
\end{equation}
we say that a number $\hat{\lambda} \in \R$ is an eigenvalue of $\l(-\Delta_r,W^{1,r}_0(\Omega)\r)$ if problem (\ref{eigenvalue_problem}) possesses a nontrivial solution $\hat{u} \in W^{1,p}_0(\Omega)$ which is said to be an eigenfunction corresponding to the eigenvalue $\hat{\lambda}$.
The set of all eigenvalues of \eqref{eigenvalue_problem} is denoted by $\hat{\sigma}(r)$ and it is known that $\hat{\sigma}(r)$ has a smallest element $\hat{\lambda}_1(r)$ which has the following properties:
\begin{itemize}
  \item[$\bullet$]
    $\hat{\lambda}_1(r)$ is positive;
  \item[$\bullet$]
    $\hat{\lambda}_1(r)$ is isolated, that is, there exists $\varepsilon>0$ such that $\left(\hat{\lambda}_1(r),\hat{\lambda}_1(r)+\varepsilon \right) \cap \hat{\sigma}(r)=\emptyset$;
  \item[$\bullet$]
    $\hat{\lambda}_1(r)$ is simple, that is, if $u,v$ are two eigenfunctions corresponding to $\hat{\lambda}_1(r)$, then $u=k v$ for some $k \in \R\setminus\{0\}$;
  \item[$\bullet$]
    \begin{align}\label{eigenvalue_infimum}\hat{\lambda}_1(r)=\inf \left [ \frac{\|\nabla u\|_{L^r(\Omega)}^r}{\|u\|_{L^r(\Omega)}^r}: u \in W^{1,r}_0(\Omega), u \neq 0 \right ].\end{align}
\end{itemize}

The infimum in \eqref{eigenvalue_infimum} is realized on the one dimensional eigenspace corresponding to $\hat{\lambda}_1(r)>0$. In what follows we denote by $\hat{u}_1(r)$ the $L^r$-normalized eigenfunction (i.e. $\|\hat{u}_1(r)\|_{L^r(\Omega)}=1$) associated to $\hat{\lambda}_1(r)$. From the representation in \eqref{eigenvalue_infimum} we easily see that $\hat{u}_1(r)$ does not change sign in $\Omega$ and so we may assume that $\hat{u}_1(r) \geq 0$. The nonlinear regularity theory implies that $\hat{u}_1(r) \in C^1_0(\close)$ and the usage of Vazquez's strong maximum principle \cite{1984-Vazquez} provides that $\hat{u}_1(r) \in \interior$.

As a consequence of the properties above we have the following simple lemma (see Papageorgiou-Kyritsi Yiallourou \cite[p. 356]{2009-Papageorgiou-Kyritsi-Yiallourou}).

\begin{lemma}\label{lemma_aux}
    Let $\eta \in \Linf_+$ be such that $\eta(x) \leq \hat{\lambda}_1(p)$ a.e. in $\Omega$ and $\eta \neq \hat{\lambda}_1(p)$. Then there exists a positive number $\kappa$ such that
    \begin{align*}
	\|\nabla u\|^p_{L^p(\Omega)}-\into \eta(x)|u(x)|^pdx \geq \kappa \|\nabla u\|^p_{L^p(\Omega)} \quad \text{for all } u \in W^{1,p}_0(\Omega).
    \end{align*}
\end{lemma}

The Lusternik-Schnirelmann minimax scheme produces a strictly increasing sequence $\left(\hat{\lambda}_k(r) \right)_{k \geq 1}$ of eigenvalues such that $\hat{\lambda}_k(r) \to + \infty$ as $k \to \infty$. We do not know if this sequence exhausts the whole spectrum of $(-\Delta_r, W^{1,r}_0(\Omega))$ but if $N=1$ (ordinary differential equations) or if $r=2$ (linear eigenvalue problem), then the Lusternik-Schnirelmann sequence of eigenvalues is the whole spectrum. In the case $r=2$ we denote by $E\l(\hat{\lambda}_k(2)\r), k \geq 1,$ the eigenspace corresponding to the eigenvalue $\hat{\lambda}_k(2)$ and we have a direct sum decomposition of the form
\begin{align*}
    H^1_0(\Omega)=\overline{\bigoplus_{k \geq 1} E\l(\hat{\lambda}_k(2)\r)}.
\end{align*}

Next, let us recall some basic definitions and facts about Morse theory. Let $X$ be a Banach space and let $(Y_1,Y_2)$ be a topological pair such that $Y_2 \subseteq Y_1 \subseteq X$. For every integer $k \geq 0$ the term $H_k(Y_1,Y_2)$ stands for the $k \overset{th}{=}$-relative singular homology group with integer coefficients. Note that $H_k(Y_1,Y_2)=0$ for all $k<0$. 
Given $\varphi \in C^1(X)$ and $c \in \R$, we introduce the following sets:
\begin{align*}
  \begin{aligned}
    & \varphi^c=\{u \in X: \varphi(u) \leq c\}\qquad && \text{(the sublevel set of $\varphi$ at $c$)},\\
    & K_\varphi=\{u\in X: \varphi'(u)=0\} && \text{(the critical set of $\varphi$ at $c$)},\\
    & K^c_\varphi=\{u \in K_\varphi: \varphi(u)=c\}&& \text{(the critical set of $\varphi$ at the level $c$)}.
  \end{aligned}
\end{align*}

For every isolated critical point $u \in K^c_\varphi$ the critical groups of $\varphi$ at $u \in K^c_\varphi$ are defined by
\begin{align*}
    C_k(\varphi,u)=H_k(\varphi^c \cap U, \varphi^c \cap U \setminus \{u\}) \quad \text{for all } k \geq 0,
\end{align*}
where $U$ is a neighborhood of $u$ such that $K_\varphi \cap \varphi^c \cap U=\{u\}$. The excision property of singular homology theory implies that the definition of critical groups above  is independent of the particular choice of the neighborhood $U$.

Suppose that $\varphi \in C^1(X)$ satisfies the C-condition and that $\inf \varphi(K_\varphi)>-\infty$. Let $c<\inf \varphi(K_\varphi)$. The critical groups of $\varphi$ at infinity are defined by
\begin{align}\label{critical_groups_at_infinity}
    C_k(\varphi,\infty)=H_k(X,\varphi^c) \quad \text{for all } k \geq 0
\end{align}
(see Bartsch-Li \cite{Bartsch-Li-1997}). This definition is independent of the choice of the level $c<\inf \varphi(K_\varphi)$ which is a consequence of the second deformation theorem (see, for example, Gasi{\'n}ski-Papageorgiou \cite[p. 628]{2006-Gasinski-Papageorgiou}).

We now assume that $K_\varphi$ is finite and introduce the following series in $t \in \R$:
\begin{align*}
    & M(t,u)=\sum_{k \geq 0} \rank C_k(\varphi,u)t^k \quad \text{for all }u \in K_\varphi,\\
    & P(t,\infty)=\sum_{k \geq 0} \rank C_k(\varphi,\infty)t^k.
\end{align*}

Then, the Morse relation (see \cite[Theorem 5.1.29]{2005-Chang}) reads as follows:
\begin{align}\label{morse_relation}
    \sum_{u \in K_\varphi}M(t,u)=P(t,\infty)+(1+t)Q(t) \quad \text{for all } t\in \R,
\end{align}
with $Q(t)$ being a formal series in $t \in \R$ with nonnegative integer coefficients.

Suppose next that $X=H$ is a Hilbert space and let $U$ be a neighborhood of a given point $x \in H$. We further assume that $\ph \in C^2(U)$, $K_\ph$ is finite and $u \in K_\ph$. The Morse index of $u$, denoted by $\mu=\mu(u)$, is defined to be the supremum of the dimensions of the vector subspaces of $H$ on which $\ph''(u) \in \mathcal{L}(H)$ is negative definite. The nullity of $u$, denoted by $\nu=\nu(u)$, is defined to be the dimension of $\ker \ph''(U)$. We say that $u \in K_\ph$ is nondegenerate if $\ph''(u)$ is invertible, that is, $\nu=\nu(u)=0$. At a nondegenerate critical point $u$ we have
\begin{align*}
    C_k(\ph,u)=\delta_{k,\mu} \Z \quad \text{for all }k \geq 0,
\end{align*}
where $\delta_{k,\mu}$ stands for the well-known Kronecker symbol.


Finally, before closing this preparatory section, let us fix our notation. Throughout this paper we denote the norm of $W^{1,p}_0(\Omega)$ through $\|\cdot\|_{W^{1,p}_0(\Omega)}$ and thanks to the Poincar\'{e} inequality it holds $\|u\|_{W^{1,p}_0(\Omega)}=\|\nabla u\|_{L^p(\Omega)}$ for all $u \in W^{1,p}_0(\Omega)$. The norm of $\R^N$ is denoted by $\|\cdot\|$ and $(\cdot,\cdot)_{\RN}$ stands for the inner product of $\R^N$.

For $s \in \R$, we set $s^{\pm}=\max\{\pm s,0\}$ and for $u \in W^{1,p}_0(\Omega)$ we define $u^{\pm}(\cdot)=u(\cdot)^{\pm}$. It is well known that
\begin{align*}
    u^{\pm} \in W^{1,p}_0(\Omega), \quad |u|=u^++u^-, \quad u=u^+-u^-.
\end{align*}
The Lebesgue measure on $\R^N$ will be stated by $|\cdot|_N$. Finally, for any measurable function $h: \Omega \times \R \to \R$ we define the Nemytskij operator $N_h : L^p(\Omega) \to (L^p(\Omega))^*$ corresponding to the function $h$ by
\begin{align*}
    N_h(u)(\cdot)=h(\cdot,u(\cdot)).
\end{align*}

\section{Three nontrivial solutions}\label{section_three}

In this section, using a combination of variational and Morse theoretic methods, we prove a multiplicity theorem producing three nontrivial solutions for problem \eqref{problem} when the reaction $f(x,\cdot)$ is $(p-1)$-superlinear but does not necessarily satisfies the Ambrosetti-Rabinowitz condition. Our result in this section improves significantly the well-known multiplicity theorem of Wang \cite{1991-Wang}. 

First we slightly strengthen the assumptions on the map $a(\cdot)$.

\begin{enumerate}[leftmargin=1.2cm]
    \item[H(a)$_1$:]
	$a(\xi)=a_0(\|\xi\|)\xi$ for all $\xi \in \R^N$ with $a_0(t)>0$ for all $t>0$, hypotheses H(a)$_1$(i)--(iii) are the same as the corresponding hypotheses H(a)(i)--(iii) and
	\begin{enumerate}
	    \item[(iv)]
		$p G_0(t)-t^2 a_0(t) \geq -c_6$ and $t^2 a_0(t)-G_0(t) \geq \hat{\eta} t^p$ for all $t>0$ and for some $c_6, \hat{\eta}>0$.
	\end{enumerate}
\end{enumerate}

\begin{remark}
    Note that the examples given in Example \ref{example_a} satisfy this new condition stated in H(a)$_1$(iv).
\end{remark}

The hypotheses on the perturbation $f$ are the following:
\begin{enumerate}
  \item[H(f)$_1$:]
    $f: \Omega \times \R \to \R$ is a Carath\'{e}odory function with $f(x,0)=0$ for a.a. $x \in \Omega$
    such that
    \begin{enumerate}
      \item[(i)]
	$|f(x,s)| \leq a(x) \left(1+|s|^{r-1}\right)$ for a.a. $x \in \Omega$, for all $s \in \R$, with $a \in \Linf_+$ and $p < r<p^*$;
      \item[(ii)]
	if $F(x,s)=\int_0^sf(x,t)dt$, then
	\begin{align*}
	  \lim_{s \to \pm \infty} \frac{F(x,s)}{|s|^{p}}=+\infty \quad  \text{uniformly for a.a. } x \in \Omega;
	\end{align*}
      \item[(iii)]
	there exist $\tau \in \left((r-p)\max \left \{ \frac{N}{p},1\right \},p^*\right)$ and $\beta_0>0$ such that
	\begin{align*}
	  \liminf_{s \to \pm \infty} \frac{f(x,s)s-pF(x,s)}{|s|^{\tau}} \geq \beta_0 \quad \text{uniformly for a.a. } x \in \Omega;
	\end{align*}
      \item[(iv)]
	there exists $\eta \in \Linf_+$ with $\eta(x) \leq \frac{c_1}{p-1}\hat{\lambda}_1(p)$ a.e. in $\Omega$ and $\eta \neq \frac{c_1}{p-1}\hat{\lambda}_1(p)$ such that
	\begin{align*}
	    \limsup_{s \to 0} \frac{pF(x,s)}{|s|^p} \leq \eta(x) \quad \text{uniformly for a.a. } x \in \Omega;
	\end{align*}
      \item[(v)]
	for every $\varrho>0$ there exists $\kappa_\varrho>0$ such that
	\begin{align*}
	    f(x,s)s+\kappa_\varrho|s|^p \geq 0 \quad \text{for a.a. }x \in \Omega \text{ and all }|s| \leq \varrho.
	\end{align*}
    \end{enumerate}
\end{enumerate}

\begin{remark}
    Hypothesis H(f)$_1$(ii) amounts to the superlinearity of the primitive $F(x,\cdot)$. This condition together with H(f)$_1$(iii) implies that $f(x,\cdot)$ is $(p-1)$-superlinear. We point out that
    the assumptions in H(f)$_1$(ii), (iii) are weaker than the Ambrosetti-Rabinowitz condition (see \eqref{AR1}, \eqref{AR2}) which is the usual hypothesis when dealing with superlinear problems 
    (see for example Wang \cite{1991-Wang}). Indeed, assume that $f(x,\cdot)$ satisfies the Ambrosetti-Rabinowitz condition and note that we may suppose $(r-p)\max \left \{ \frac{N}{p},1\right \}<\mu$. Hence, we have
    \begin{align*}
	\frac{f(x,s)-pF(x,s)}{|s|^\mu}
	& =\frac{f(x,s)s-\mu F(x,s)}{|s|^\mu}+\frac{(\mu-p)F(x,s)}{|s|^\mu}\\
	& \geq (\mu-p) \eta \quad \text{for all }x\in \Omega \text{ and for all }|s|\geq M
    \end{align*}
    (see \eqref{AR1} and \eqref{AR3}).
\end{remark}

\begin{example}
    For the sake of simplicity we drop the $x$-dependence and consider the following two functions satisfying hypotheses $H(f)_1$:
    \begin{align*}
	& f_1(s)=
	\begin{cases}
	  \eta s^p \quad & \text{if } |s| \leq 1,\\
	  \eta s^r & \text{if } |s|>1
	\end{cases}
	\quad \text{with }\eta \in (0, \hat{\lambda}_1(p)) \text{ and }p<r<p^*;\\
	& f_2(s)=|s|^{p-2}s \ln (1+|s|).
    \end{align*}
    Note that $f_1$ satisfies the Ambrosetti-Rabinowitz condition but $f_2$ does not.
\end{example}

Let $\varphi: W^{1,p}_0(\Omega) \to \R$ be the energy functional of problem (\ref{problem}) given by
\begin{align*}
  \varphi(u)=\into G(\nabla u)dx- \into F(x,u) dx,
\end{align*}
which is of class $C^1$. Furthermore, we define the positive and negative truncations of $f(x,\cdot)$, namely $f_{\pm}(x,s)=f(x,\pm s^{\pm})$, and consider the $C^1$-functionals $\varphi_{\pm}: W^{1,p}_0(\Omega) \to \R$ defined by
\begin{align*}
    \varphi_{\pm}(u)=\into G(\nabla u)dx- \into F_{\pm}(x,u) dx,
\end{align*}
with $F_{\pm}(x,s)=\int^s_0 f_{\pm}(x,t)dt$.

\begin{proposition}\label{proposition_C_condition}
    If H(a)$_1$ and H(f)$_1$ are satisfied, then the functionals $\varphi$ and $\varphi_{\pm}$ fulfill the $C$-condition.
\end{proposition}

\begin{proof}
    We start with the proof for $\varphi_+$. To this end, let $(u_n)_{n\geq 1} \subseteq W^{1,p}_0(\Omega)$ be a sequence such that
    \begin{align}\label{C_cond_1}
	|\varphi_+(u_n)| \leq M_1\quad \text{ for all }n \geq 1
    \end{align}
    with some $M_1>0$ and
    \begin{align}\label{C_cond_2}
	\l(1+\|u_n\|_{\W1p0}\r)\varphi_+'(u_n) \to 0 \text{ in }W^{-1,p'}(\Omega).
    \end{align}
    By means of (\ref{C_cond_2}) we obtain
    \begin{align*}
	|\lan \varphi_+'(u_n),v| \leq \frac{\eps_n \|v\|_{W^{1,p}_0(\Omega)}}{1+\|u_n\|_{W^{1,p}_0(\Omega)}}
    \end{align*}
    for all $v \in \W1p0$ and $\eps_n \searrow 0$ which means that
    \begin{align}\label{C_cond_3}
	\left|\into\left(a(\nabla u_n),\nabla v\right)_{\R^N}dx-\into f_+(x,u_n)vdx \right| \leq \frac{\eps_n \|v\|_{W^{1,p}_0(\Omega)}}{1+\|u_n\|_{W^{1,p}_0(\Omega)}}
    \end{align}
    for all $n \geq 1$. Acting on \eqref{C_cond_3} with $v=-u_n^- \in \W1p0$ and applying Lemma \ref{lemma_properties}(c) yields
    \begin{align*}
	\frac{c_1}{p-1} \|\nabla u_n^- \|_{\W1p0}\leq \eps_n,
    \end{align*}
    for all $n \geq 1$ which means that
    \begin{align}\label{C_cond_4_b}
	u_n^- \to 0 \quad \text{in } \W1p0 \text{ as }n \to +\infty.
    \end{align}
    Then, from (\ref{C_cond_1}) and \eqref{C_cond_4_b} we obtain
    \begin{align}\label{C_cond_5}
	\into p G(\nabla u_n^+)dx-\into pF(x,u_n^+)dx \leq M_2,
    \end{align}
    with some $M_2>0$. Taking $v=u_n^+ \in \W1p0$ in (\ref{C_cond_3}) gives 
    \begin{align}\label{C_cond_4}
	-\into\left(a(\nabla u_n^+),\nabla u_n^+\right)_{\R^N}dx +\into f(x,u_n^+)u_n^+dx  \leq \eps_n,
    \end{align}
    for all $n \geq 1$. Now, adding (\ref{C_cond_5}) and (\ref{C_cond_4}), we get
    \begin{align}\label{C_cond_6}
      \begin{split}
	M_3
	& \geq \into \left[pG(\nabla u_n^+)- \left(a(\nabla u_n^+),\nabla u_n^+\right)_{\R^N}\right]dx\\
	& \qquad +\into \left[f(x,u_n^+)u_n^+ -pF(x,u_n^+)\right]dx,
      \end{split}
    \end{align}
    for all $n \geq 1$ and some $M_3>0$. By virtue of hypothesis H(a)$_1$(iv) we derive from (\ref{C_cond_6})
    \begin{align}\label{C_cond_7}
      \begin{split}
	\into \left(f(x,u_n^+)u_n^+ -pF(x,u_n^+)\right)dx \leq M_4.
      \end{split}
    \end{align}
    Taking into account hypotheses H(f)$_1$(i) and (iii), there is a number $\beta_1 \in \l(0,\beta_0\r)$ and a constant $M_5>0$ such that
    \begin{align}\label{C_cond_8}
	\beta_1 |s|^\tau-M_5\leq f(x,s)s-pF(x,s) \quad \text{for a.a. }x\in \Omega \text{ and for all }s \geq 0.
    \end{align}
    Combining (\ref{C_cond_7}) and \eqref{C_cond_8} gives
    \begin{align}\label{C_cond_9}
	\l(u_n^+\r)_{n\geq 1} \text{ is bounded in } \Lp{\tau}.
    \end{align}
    
    Let us first consider the case $N> p$. Without loss of generality we may suppose that $1<\tau\leq r <p^*$ (cf. hypothesis H(f)$_1$(iii)). Then, we find a number $t \in [0,1)$ such that
    \begin{align}\label{interpolation}
	\frac{1}{r}=\frac{1-t}{\tau}+\frac{t}{p^*}
    \end{align}
    and the usage of the interpolation theory implies that
    \begin{align}\label{C_cond_10}
	\l\|u_n^+\r\|_{L^r(\Omega)} \leq \l\|u_n^+\r\|^{1-t}_{L^\tau(\Omega)}\l\|u_n^+\r\|^t_{L^{p^*}(\Omega)}
    \end{align}
    (see Gasi{\'n}ski-Papageorgiou \cite[p. 905]{2006-Gasinski-Papageorgiou}). Combining \eqref{C_cond_9}, \eqref{C_cond_10}, and the Sobolev embedding theorem yields
    \begin{align}\label{C_cond_11}
	\l\|u_n^+\r\|_{L^r(\Omega)}^r \leq M_6 \l\|u_n^+\r\|^{tr}_{W^{1,p}_0(\Omega)} \quad \text{for all }n \geq 1
    \end{align}
    with some positive constant $M_6$. Applying again $v=u_n^+$ in (\ref{C_cond_3}) one has
    \begin{align}\label{C_cond_12}
	\left|\into\left(a\l(\nabla u_n^+\r),\nabla u_n^+\right)_{\R^N}dx-\into f(x,u_n^+)u_n^+dx \right| \leq \eps_n \quad \text{for all }n  \geq 1.
    \end{align}
    Taking into account the growth condition of hypothesis H(f)$_1$(i) we infer
    \begin{align}\label{C_cond_13}
	f(x,s)s\leq \hat{a}(x)+M_7|s|^r \quad \text{for a.a. }x \in \Omega, \text{ for all }s\in \R,
    \end{align}
    with $\hat{a}\in \Linf$ and $M_7>0$. With the aid of Lemma \ref{lemma_properties}(c) and \eqref{C_cond_13} we obtain from \eqref{C_cond_12}
    \begin{align*}
	\frac{c_1}{p-1}\l\|\nabla u_n^+\r\|^p_{L^p(\Omega)} \leq M_8\left(1+\l\|u_n^+\r\|^r_{L^r(\Omega)}\right) \quad \text{for all } n \geq 1
    \end{align*}
    with $M_8>0$. This estimate in conjunction with \eqref{C_cond_11} yields
    \begin{align}\label{C_cond_14}
	\l\|u_n^+\r\|^p_{W^{1,p}_0(\Omega)}\leq M_9 \left (1+\l\|u_n^+\r\|_{W^{1,p}_0(\Omega)}^{tr} \right) \quad \text{ for all }n\geq 1
    \end{align}
    and for some $M_{9}>0$. Taking into account the choice of $\tau$ (see hypothesis H(f)$_1$(iii)) and relation \eqref{interpolation} we see that $tr<p$ which implies that $\l(u_n^+\r)_{n\geq 1} \subseteq W^{1,p}_0(\Omega)$ is bounded (see \eqref{C_cond_14}).

    Now, let $N\leq p$ and note that in this case we have $p^*=\infty$ and the Sobolev embedding theorem gives $W^{1,p}_0(\Omega) \subseteq L^{\tilde{q}}(\Omega)$ for all $\tilde{q}\in [1,+\infty)$. Let $\hat{q}$ be a number such that $ 1<\tau\leq r<\hat{q}$. As before, we find $t \in [0,1)$ such that
    \begin{align*}
	\frac{1}{r}=\frac{1-t}{\tau}+\frac{t}{\hat{q}}.
    \end{align*}
    Hence
    \begin{align*}
	tr=\frac{\hat{q}(r-\tau)}{\hat{q}-\tau}.
    \end{align*}
    Moreover, we observe that
    \begin{align}\label{C_cond_15}
	tr=\frac{\hat{q}(r-\tau)}{\hat{q}-\tau} \to r-\tau \quad \text{as } \hat{q} \to +\infty=p^*.
    \end{align}
    By the choice of $\tau$ and since $N\leq p$ we have $r-\tau<p$. Combining this fact with \eqref{C_cond_15} we see that $tr<p$ if $\hat{q}$ is chosen large enough. Now we may apply the same arguments as in the case $N>p$ where $p^*$ is replaced by $\hat{q}>r$ sufficiently large. This yields the boundedness of the sequence $\l(u_n^+\r)_{n\geq 1}$ in $W^{1,p}_0(\Omega)$ in the case $N\leq p$ as well.
    We have shown in both cases that $\l(u_n^+\r)_{n\geq 1}$ is bounded in $\W1p0$ and due to \eqref{C_cond_4_b} we have that $\l(u_n\r)_{n\geq 1}$ is bounded in $\W1p0$ as well.
    Now we may suppose that (for a subsequence if necessary)
    \begin{align}\label{C_cond_16}
	u_n \rightharpoonup u \text{ in } W^{1,p}_0(\Omega) \quad \text{and} \quad u_n \to u \text{ in }L^p(\Omega).
    \end{align}
    Using again \eqref{C_cond_3} with the special choice $v=u_n-u$ and passing to the limit as $n$ goes to $+\infty$, we derive, thanks to \eqref{C_cond_16},
    \begin{align*}
	\lim_{n\to \infty}\left \lan A(u_n),u_n-u \right\ran =0.
    \end{align*}
    Since $A$ satisfies the (S)$_+$-property (see Proposition \ref{proposition_basic_properties}) we finally conclude
    \begin{align*}
	u_n \to u \text{ in } W^{1,p}_0(\Omega).
    \end{align*}
    This proves that $\varphi$ fulfills the $C$-condition. Analogously, applying similar arguments, one can prove the same result for the functionals $\varphi$ and $\varphi_{-}$. That finishes the proof.
\end{proof}

Now we are going to show that the functionals $\varphi$ and $\varphi_{\pm}$ satisfy the mountain pass geometry.

\begin{proposition}\label{proposition_phi_phi_plus_minus_zero}
    Assume H(a)$_1$ and H(f)$_1$, then $u=0$ is a local minimizer of the functionals $\varphi$ and $\varphi_{\pm}$.
\end{proposition}

\begin{proof}
    We only show this proposition for $\varphi_+$, the proofs for $\varphi$ and $\varphi_{-}$ can be done similarly. By means of hypothesis H(f)$_1$(iv) we find for every $\varepsilon>0$ a number $\delta=\delta(\varepsilon) >0$ such that
    \begin{align}\label{1}
	F(x,s) \leq \frac{1}{p}(\eta(x)+\varepsilon)|s|^p \quad \text{for a.a. }x \in \Omega \text{ and for all }|s| \leq \delta.
    \end{align}
    Let $u \in C^1_0(\overline{\Omega})$ be such that $\|u\|_{C^1_0(\overline{\Omega})} \leq \delta$. With regards to Corollary \ref{corollary_upper_lower_estimates}, (\ref{1}), Lemma \ref{lemma_aux}, and \eqref{eigenvalue_infimum} we obtain
    \begin{align}\label{u=0_local_minimizer}
	\begin{split}
	    \varphi_+(u)
	    & = \into G(\nabla u)dx -\into F_+(x,u)dx\\
	    & \geq \frac{c_1}{p(p-1)}\|\nabla u\|^p_{L^p(\Omega)}-\frac{1}{p}\into \eta(x)\left(u^+\right)^p dx -\frac{\eps}{p} \l\|u^+\r\|^p_{L^p(\Omega)}\\
	    & \geq \frac{1}{p}\left(\frac{c_1}{p-1}\|\nabla u\|^p_{L^p(\Omega)}-\into \eta(x)|u|^p dx\right) -\frac{\eps}{p\hat{\lambda}_1(p)} \|\nabla u\|^p_{L^p(\Omega)}\\
	    & \geq \frac{1}{p} \left(\kappa-\frac{\eps}{\hat{\lambda}_1(p)} \right) \|\nabla u\|^p_{L^p(\Omega)}.
	\end{split}
    \end{align}
    Choosing $\eps>0$ small enough such that $\eps<\left(0,\kappa \hat{\lambda}_1(p)\right)$ we see from (\ref{u=0_local_minimizer}) that
    \begin{align*}
	\varphi_+(u)\geq 0=\varphi_+(0) \quad \text{for all } u \in C^1_0(\overline{\Omega}) \text{ with }0\leq \|u\|_{C^1_0(\overline{\Omega})} \leq \delta.
    \end{align*}
    This implies that $u=0$ is a local $C^1_0(\overline{\Omega})$-minimizer of $\varphi_+$. Invoking Proposition \ref{proposition_local_minimizers} yields that $u=0$ is a local $W^{1,p}_0(\Omega)$-minimizer of $\varphi_+$ as well.
\end{proof}

It is easy to see that the critical points of $\varphi_+$ (resp. of $\varphi_-$) are positive (resp. negative). Therefore, we may assume that $u=0$ is an isolated critical point of the functionals $\varphi_{\pm}$, otherwise there would exist a sequence of distinct positive, resp. negative, solutions of \eqref{problem}. 

Consequently, we find small numbers $\varrho_{\pm} \in (0,1)$ such that
\begin{align}\label{eq_0}
    \inf \left\{\varphi_{\pm}(u):\|u\|_{W^{1,p}_0(\Omega)}=\varrho_{\pm}\right\}=:m_{\pm}>0=\varphi_{\pm}(0)
\end{align}
(see Aizicovici-Papageorgiou-Staicu \cite[Proof of Proposition 29]{2008-Aizicovici-Papageorgiou-Staicu}).

Now we are going to prove the existence of two constant sign solutions of problem \eqref{problem}.

\begin{proposition}\label{proposition_constant_sign_solutions}
    Under the assumptions H(a)$_1$ and H(f)$_1$ problem (\ref{problem}) possesses at least two constant sign solutions $u_0 \in \ints \left(C^1_0(\overline{\Omega})_+\right)$ and $v_0 \in - \ints \left(C^1_0(\overline{\Omega})_+\right)$.
\end{proposition}

\begin{proof}
    We start with the proof of the existence of the positive solution. Recall that $\hat{u}_1(p) \in \ints \left(C^1_0(\overline{\Omega})_+\right)$ denotes the $L^p$-normalized (i.e. $\|\hat{u}_1(p)\|_{L^p(\Omega)}=1$) eigenfunction corresponding to the first eigenvalue $\hat{\lambda}_1(p)$ of $\left ( -\Delta_p, W^{1,p}_0(\Omega) \right)$. First, we show that
    \begin{align}\label{eq_303}
	\varphi_+(t\hat{u}_1(p))\to -\infty \quad \text{as }t\to+\infty.
    \end{align}

    By means of hypotheses H(f)$_1$(i) and (ii), for every $\eps>0$ there exists a constant $M_{10}=M_{10}(\eps)>0$ such that
    \begin{align}\label{eq_2}
	F(x,s) \geq \eps |s|^p-M_{10} \quad \text{for a.a. }x\in \Omega \text{ and for all }s\in \R.
    \end{align}
    From Corollary \ref{corollary_upper_lower_estimates} and \eqref{eq_2} we obtain for $t>0$
    \begin{align}\label{eq_3}
      \begin{split}
	\varphi_+\left(t \hat{u}_1(p)\right)
	& \leq c_5 |\Omega|_N+ t^p \|\nabla \hat{u}_1(p)\|^p_{\Lp{p}}-\eps t^p+M_{10}|\Omega|_N\\
	& = t^p\left(\hat{\lambda}_1(p)-\eps\right)+(c_5+M_{10})|\Omega|_N.
      \end{split}
    \end{align}
    Choosing $\eps>\hat{\lambda}_1(p)$ in \eqref{eq_3} and letting $t \to +\infty$ implies \eqref{eq_303}.

    Taking into account \eqref{eq_303} and \eqref{eq_0} we find a number $t>0$ large enough such that
    \begin{align}\label{eq_6}
	\varphi_+\l(t\hat{u}_1(p)\r) \leq \varphi_+(0)=0<m_+ \quad \text{and} \quad \varrho_{+}<\l\|t\hat{u}_1(p)\r\|_{W^{1,p}_0(\Omega)}.
    \end{align}
    Thanks to \eqref{eq_0}, \eqref{eq_6} and Proposition \ref{proposition_C_condition} we may apply Theorem \ref{theorem_mountain_pass} (mountain pass theorem) which provides the existence of an element $u_0 \in W^{1,p}_0(\Omega)$ such that
    \begin{align}\label{eq_4}
	\varphi_+(0)=0<m_+\leq \varphi_+(u_0) \quad \text{and} \quad \varphi_+'(u_0)=0.
    \end{align}
    The first relation in \eqref{eq_4} ensures that $u_0 \neq 0$ and the second one results in
    \begin{align}\label{eq_7}
	\left \lan A u_0,v \right\ran = \left\lan N_{f_+}(u_0),v\right \ran\quad \text{for all }v \in W^{1,p}_0(\Omega).
    \end{align}
    Choosing $v=-u_0^-$ as test function in (\ref{eq_7}) gives
    \begin{align}\label{eq_8}
	\into \left(a(\nabla u_0),-\nabla u_0^-\right)_{\R^N} dx=0.
    \end{align}
    Combining (\ref{eq_8}) and Lemma \ref{lemma_properties}(c) we have
    \begin{align*}
	\frac{c_1}{p-1} \l\|\nabla u_0^-\r\|^p_{L^p(\Omega)} \leq 0.
    \end{align*}
    Hence,
    \begin{align*}
	u_0 \geq 0, u_0 \neq 0.
    \end{align*}
    Then, \eqref{eq_7} becomes
    \begin{equation*}
      \begin{aligned}
	-\divergenz a(\nabla u_0(x)) & = f(x,u_0(x)) \quad && \text{in } \Omega,\\
	u & = 0  &&\text{on } \partial \Omega.
      \end{aligned}
    \end{equation*}
    From the nonlinear regularity theory we obtain $u_0 \in \Linf$ (see Ladyzhenskaya-Ural{\cprime}tseva \cite[p. 286]{1968-Ladyzhenskaya-Ural'tseva}) and then $u_0 \in C^{1}_0(\close)$ (see Lieberman \cite{1991-Lieberman}). By means of hypothesis H(f)$_1$(v) we find, for $\varrho=\|u_0\|_{C^(\close)}$, a constant $\kappa_\varrho >0$ such that
    \begin{align*}
	-\divergenz a(\nabla u_0(x))+\kappa_\varrho u_0(x)^{p-1}=f(x,u_0(x))+\kappa_\varrho u_0(x)^{p-1} \geq 0 \quad \text{for a.a. }x\in \Omega.
    \end{align*}
    Hence,
    \begin{align}\label{pucci1}
	\divergenz a(\nabla u_0(x))\leq \kappa_\varrho u_0(x)^{p-1} \quad \text{for a.a. }x\in \Omega.
    \end{align}
    Let $\gamma(t)=ta_0(t)$ for $t>0$. We have
    \begin{align}\label{pucci3}
	t\gamma'(t)=t^2a_0'(t)+ta_0(t).
    \end{align}
    Integration by parts and applying H(a)$_1$(iv) yields
    \begin{align}\label{pucci2}
	\int_0^t s \gamma'(s) ds =t \gamma(t)-\int_0^t \gamma(s)ds=t^2 a_0(t)-G_0(t) \geq \hat{\eta}t^p.
    \end{align}
    Then, due to \eqref{pucci1} and \eqref{pucci2}, we may apply the strong maximum principle of Pucci-Serrin \cite[p. 111]{2007-Pucci-Serrin} which implies that $u_0(x)>0$ for all $x \in \Omega$. In addition, the boundary point theorem of Pucci-Serrin \cite[p. 120]{2007-Pucci-Serrin} yields $u_0 \in \interior$.
    
    Using similar arguments one could easily verify the assertion for the existence of the constant sign solution $v_0 \in -\interior$ working with the functional $\ph_-$ instead of $\ph_+$.
    \end{proof}

    Now, we are interested to find a third nontrivial solutions of (\ref{problem}) via Morse theory. To this end, we will compute certain critical groups of the functionals $\varphi$ and $\varphi_{\pm}$. We start with the computation of the critical groups of $\varphi$ at infinity.

\begin{proposition}\label{proposition_phi_infinity}
    Assume H(a)$_1$ and H(f)$_1$, then $C_k(\varphi,\infty)=0$ for all $k \geq 0$.
\end{proposition}

\begin{proof}
    By means of H(f)$_1$(i) and (ii), for every $\eps>0$, there exists a constant $M_{11}>0$ such that
    \begin{align}\label{6}
	F(x,s) \geq \eps|s|^p-M_{11} \quad \text{for a.a. }x \in \Omega \text{ and for all } s\in \R.
    \end{align}
    By virtue of Corollary \ref{corollary_upper_lower_estimates} and (\ref{6}) there holds for $u \in \W1p0 \setminus \{0\}$ and for every $t>0$
    \begin{align*}
	\begin{split}
	    \varphi(tu)
	    & = \into G(t \nabla u) dx-\into F(x,tu)dx\\
	    & \leq c_5 \left (|\Omega|_N+t^p\|\nabla u\|^p_{\Lp{p}} \right)-\eps t^p\| u\|_{L^p(\Omega)}^p+M_{11}|\Omega|_N\\
	    & = t^p \left(c_5\|\nabla u\|^p_{\Lp{p}} -\eps\| u\|_{L^p(\Omega)}^p \right)+M_{12},
	\end{split}
    \end{align*}
    with $M_{12}=\left(c_5 +M_{11} \right) |\Omega_N| $. Choosing $\eps>\frac{c_5 \|\nabla u\|^p_{\Lp{p}}}{\| u\|^p_{L^p(\Omega)}}$ implies that
    \begin{align}\label{9}
	\begin{split}
	    \varphi(tu) \to -\infty \quad \text{as }t \to +\infty.
	\end{split}
    \end{align}
    Thanks to the hypotheses H(f)$_1$(i) and (iii), there is a number $\beta_2 \in (0,\beta_0)$ and a constant $M_{13}>0$ such that
    \begin{align}\label{8}
	pF(x,s)-f(x,s)s \leq M_{13}-\beta_2 |s|^\tau  \quad \text{for a.a. }x\in \Omega \text{ and for all }s \in \R.
    \end{align}
    Taking into account hypothesis H(a)$_1$(iv) and \eqref{8} we obtain
    \begin{align}\label{10}
	\begin{split}
	    \frac{d}{dt}\varphi(tu)
	    & = \lan \varphi'(tu),u\ran\\
	    & = \frac{1}{t} \lan \varphi'(tu),tu\ran\\
	    & = \frac{1}{t} \left[\into (a(t\nabla u),t\nabla u)_{\R^N}dx-\into f(x,tu)tudx \right]\\
	    & \leq \frac{1}{t}\left[ \into p G(t \nabla u) dx +(c_6+M_{13})|\Omega|_N-\into pF(x,tu)dx\right]\\
	    & = \frac{1}{t}\left[p \varphi(tu) + M_{14} \right]
	\end{split}
    \end{align}
    with $M_{14}=(c_6+M_{13})|\Omega|_N$. Combining \eqref{9} and \eqref{10} we conclude that
    \begin{align*}
	\frac{d}{dt}\varphi(tu)<0 \quad \text{ for }t>0 \text{ sufficiently large}.
    \end{align*}
    Therefore, for every $u \in \partial B_1=\l\{y \in \W1p0: \|y\|_{\W1p0}=1\r\}$, there exists an unique $\psi(u)>0$ such that
    \begin{align*}
	\varphi(\psi(u)u)=\varrho_*< - \frac{M_{14}}{p} 
    \end{align*}
    (see \eqref{10}). Moreover, the implicit function theorem implies that $\psi \in C(\partial B_1)$.
    
    Now we extend $\psi$ on $\W1p0 \setminus \{0\}$ by setting
    \begin{align*}
	\tilde{\psi}(u)=\frac{1}{\|u\|_{W^{1,p}_0(\Omega)}}\psi\left(\frac{u}{\|u\|_{W^{1,p}_0(\Omega)}}\right) \quad \text{for all }u \in \W1p0 \setminus \{0\}.
    \end{align*}
    It is clear that $\tilde{\psi} \in C\left(W^{1,p}_0(\Omega)\setminus \{0\}\right)$ and $\varphi\left(\tilde{\psi}(u)u\right)=\varrho_*$ for all $u \in \W1p0 \setminus \{0\}$. 
    Note that $\varphi(u)=\varrho_*$ implies $\psi(u)=1$. Then, putting
    \begin{align}\label{12}
	\hat{\psi}(u)=
	\begin{cases}
	    1 \quad & \text{if } \varphi(u) \leq \varrho_*,\\
	    \tilde{\psi}(u) & \text{if } \varphi(u)>\varrho_*,
	\end{cases}
    \end{align}
    we have $\hat{\psi} \in C\left(W^{1,p}_0(\Omega)\setminus \{0\}\right)$. 
    
    Next, we introduce the deformation $h : [0,1] \times W^{1,p}_0(\Omega)\setminus \{0\} \to W^{1,p}_0(\Omega)\setminus \{0\}$ defined by
    \begin{align*}
	h(t,u)=(1-t)u+t\hat{\psi}(u)u.
    \end{align*}
    It is easy to see that $h(0,u)=u$ and $h(1,u)\in\varphi^{\varrho_*}$ for all $u \in W^{1,p}_0(\Omega)\setminus\{0\}$. Moreover, thanks to \eqref{12} there holds
    \begin{align*}
	h(t,\cdot)\big|_{\varphi^{\varrho_*}}=\id\big|_{\varphi^{\varrho_*}} \quad \text{for all }t \in [0,1].
    \end{align*}
    This means that the sublevel set $\varphi^{\varrho_*}$ is a deformation retract of $W^{1,p}_0(\Omega)\setminus \{0\}$. Because of the radial retraction  $u \to \frac{u}{\|u\|_{\W1p0}}$ for all $u \in \W1p0 \setminus\{0\}$ we see that $\partial B_1$ is a retract of $W^{1,p}_0(\Omega)\setminus \{0\}$ while the deformation
    \begin{align*}
	h_0(t,u)=(1-t)u+t\frac{u}{\|u\|_{\W1p0}} \quad \text{for all } (t,u)\in[0,1] \times W^{1,p}_0(\Omega)\setminus \{0\},
    \end{align*}
    points out that $W^{1,p}_0(\Omega)\setminus \{0\}$ is deformable into $\partial B_1$ over $\W1p0$. Then, we may apply Theorem 6.5 of Dugundji \cite[p. 325]{1966-Dugundji} which implies that $\partial B_1$ is a deformation retract of $W^{1,p}_0(\Omega)\setminus \{0\}$.
    We conclude that $\varphi^{\varrho_*}$ and $\partial B_1$ are homotopy equivalent. Hence,
    \begin{align}\label{15}
	H_k\left(W^{1,p}_0(\Omega),\varphi^{\varrho_*}\right)=H_k\left(W^{1,p}_0(\Omega),\partial B_1\right) \quad \text{for all }k\geq 0.
    \end{align}
    Since the space $\W1p0$ is infinite dimensional, it follows that $\partial B_1$ is contractible in itself. Then, from Granas-Dugundji \cite[p. 389]{2003-Granas-Dugundji} we have
    \begin{align*}
	H_k\left(W^{1,p}_0(\Omega),\partial B_1\right)=0 \quad \text{for all }k \geq 0,
    \end{align*}
    which in view of \eqref{15} gives
    \begin{align}\label{13}
	H_k\left(W^{1,p}_0(\Omega),\varphi^{\varrho_*}\right)=0 \quad \text{for all }k \geq 0.
    \end{align}
    Choosing $\varrho_*<-\frac{M_{14}}{p}$ even smaller if necessary, we conclude from \eqref{13} that
    \begin{align*}
	C_k(\varphi,\infty)=0 \quad \text{for all } k\geq 0
    \end{align*}
    (see \eqref{critical_groups_at_infinity}).
\end{proof}

A similar reasoning leads to the following result.

\begin{proposition}\label{proposition_phi_plus_minus__infinity}
    Assume H(a)$_1$ and H(f)$_1$, then
    \begin{align*}
	C_k(\varphi_{\pm},\infty)=0 \quad \text{for all }k \geq 0.
    \end{align*}
\end{proposition}

\begin{proof}
    We do the proof only for the functional $\ph_+$, the assertion for $\ph_-$ can be done similarly.
    Let $\partial B_1^+:=\l\{u \in \partial B_1: u^+\neq 0\r\}$ and $t>0$. As in the proof of Proposition \ref{proposition_phi_infinity} we can show that for all $u \in \partial B_1^+$ there holds
    \begin{align}\label{19}
	\varphi_+(tu) \to -\infty \quad \text{as }t \to +\infty.
    \end{align}
    Taking into account H(a)$_1$(iv) and \eqref{8} yields, for all $u \in \partial B_1^+$,
    \begin{align}\label{20}
	\begin{split}
	    \frac{d}{dt}\varphi_+(tu)
	    & = \lan \varphi_+'(tu),u\ran\\
	    & = \frac{1}{t} \lan \varphi_+'(tu),tu\ran\\
	    & = \frac{1}{t} \left[\into (a(t\nabla u),t\nabla u)_{\R^N}dx-\into f_+(x,tu)tudx \right]\\
	    & \leq \frac{1}{t}\left[ \into p G(t\nabla u) dx + (c_6+M_{15}) |\Omega|_N -\into pF(x,tu^+)dx\right]\\
	    & \leq \frac{1}{t}\left[ p \ph_+(tu)+ M_{16}\right]
	\end{split}
    \end{align}
    where $M_{16}=(c_6+M_{15})|\Omega|_N$ and $M_{15}>0$. Regarding \eqref{19} and \eqref{20}, we conclude that
    \begin{align*}
	\frac{d}{dt}\varphi_+(tu)<0 \quad \text{for all }t>0 \text{ sufficiently large}.
    \end{align*}
    As before, for every $u \in \partial B_1^+$, we find an unique $\psi_+(u)>0$ such that $\varphi_+(\psi_+(u)u)=\varrho_*^+<-\frac{M_{16}}{p}$ and the implicit function theorem implies that $\psi_+ \in C(\partial B_1^+)$.

    Let $E_+=\{u \in W^{1,p}_0(\Omega): u^+ \neq 0\}$ and set for all $u\in E_+$
    \begin{align*}
	\tilde{\psi}_+(u)=\frac{1}{\|u\|_{W^{1,p}_0(\Omega)}}\psi_+\left(\frac{u}{\|u\|_{W^{1,p}_0(\Omega)}}\right).
    \end{align*}
    Obviously, $\tilde{\psi}_+ \in C\left(E_+\right)$ and $\varphi_+\left(\tilde{\psi}_+(u)u\right)=\varrho_*^+$. Moreover, if $\varphi_+(u)=\varrho_*^+$, then $\widetilde{\psi}_+(u)=1$.
    Hence,
    \begin{align}\label{22}
	\hat{\psi}_{+}(u):=
	\begin{cases}
	    1 \quad & \text{if } \varphi_+(u) \leq \varrho_*^+,\\
	    \tilde{\psi}_+(u) & \text{if } \varphi_+(u)>\varrho_*^+,
	\end{cases}
    \end{align}
    belongs to $C\left(E_+\right)$.

    Consider the deformation $h_+ : [0,1] \times E_+ \to E_+$ defined by
    \begin{align*}
	h_+(t,u)=(1-t)u+t\hat{\psi}_{+}(u)u.
    \end{align*}
    We see at once that $h_+(0,u)=u$, $h_+(1,u) \in \left(\varphi_+\right)^{\varrho_*^+}$ for all $u \in E_+$, and
    \begin{align*}
	h_+(t,\cdot)\big|_{\left(\varphi_+\right)^{\varrho_*^+}}=\id\big|_{\left(\varphi_+\right)^{\varrho_*^+}} \quad \text{for all }t \in [0,1]
    \end{align*}
    (cf. (\ref{22})). Consequently, $\left(\varphi_+\right)^{\varrho_*^+}$ is a strong deformation retract of $E_+$.

    Let us consider the deformation $\hat{h}_+: [0,1]\times E_+ \to E_+$ defined by
    \begin{align*}
	\hat{h}_+(t,u)=(1-t)u+tu_0,
    \end{align*}
    where $u_0 \in E_+$ is fixed. Then, $\hat{h}_+(0,u)=u$ and $\hat{h}_+(1,u)=u_0$ which means that $\id_{E_+}$ is homotopic to the constant map $u \mapsto u_0$. Thus, $E_+$ is contractible to itself (see Bredon \cite[Proposition 14.5]{1997-Bredon}) and from Granas-Dugundji \cite[p. 389]{2003-Granas-Dugundji}, it follows
    \begin{align*}
	H_k\left(W^{1,p}_0(\Omega),E_+\right)=0 \quad \text{for all }k \geq 0.
    \end{align*}
    Then we infer
    \begin{align}\label{25}
	H_k\left(W^{1,p}_0(\Omega),\left(\varphi_+\right)^{\varrho_*^+}\right)=0 \quad \text{for all }k \geq 0.
    \end{align}
    As before, we choose $\varrho_*^+<- \frac{M_{16}}{p}$ sufficiently small. Thus, \eqref{25} implies
    \begin{align*}
	C_k\left(\varphi_+,\infty\right)=0 \quad \text{for all }k \geq 0.
    \end{align*}
    This yields the assertion of the proposition.
\end{proof}

Recall that $u_0 \in \ints \left(C^1_0(\overline{\Omega})_+\right)$ and $v_0 \in -\ints \left(C^1_0(\overline{\Omega})_+\right)$ are the constant sign solutions of (\ref{problem}) obtained in Proposition \ref{proposition_local_minimizers}. 
We may assume that $K_{\ph}=\{0,u_0,v_0\}$, otherwise we would find another nontrivial solution of (\ref{problem}) which would belong to $C^{1}_0(\Omega)$ as a consequence of the nonlinear regularity theory (see Ladyzhenskaya-Ural{\cprime}tseva \cite{1968-Ladyzhenskaya-Ural'tseva}) and Lieberman \cite{1991-Lieberman}) and therefore we would have done.

Note that $K_{\ph}=\{0,u_0,v_0\}$ ensures that $K_{\ph_+}=\{0,u_0\}$ and $K_{\ph_-}=\{0,v_0\}$.

\begin{proposition}\label{proposition_phi_plus_minus_u0_v0}
    Assume H(a)$_1$ and H(f)$_1$, then
    \begin{align*}
	C_k(\varphi_{+},u_0)=C_k(\varphi_{-},v_0)=\delta_{k,1}\Z \quad \text{for all }k \geq 0.
    \end{align*}
 \end{proposition}

\begin{proof}
    We only compute $C_k(\varphi_{+},u_0)$, the computation of $C_k(\varphi_{-},v_0)$ is done in a similar way.
    Let $\varsigma_1, \varsigma_2 \in \R$ be two numbers such that
    \begin{align}\label{32}
	\varsigma_1<0=\varphi_{+}(0)<\varsigma_2<m_{+}\leq \varphi_{+}(u_0)
    \end{align}
    (see \eqref{eq_0} and \eqref{eq_4}) and consider the following triple of sets
    \begin{align*}
	\left(\varphi_{+}\right)^{\varsigma_1} \subseteq \left(\varphi_{+}\right)^{\varsigma_2}\subseteq W^{1,p}_0(\Omega).
    \end{align*}
    Concerning this triple of sets we study the corresponding long exact sequence of homology groups which is given by
    \begin{align}\label{33}
	\begin{split}
	& \ldots \longrightarrow H_k \left( W^{1,p}_0(\Omega), \left(\varphi_{+}\right)^{\varsigma_1} \right) \overset{i_*}{\longrightarrow} H_k \left( W^{1,p}_0(\Omega),\left(\varphi_{+}\right)^{\varsigma_2} \right)\\
	&\qquad \qquad \overset{\partial_*}{\longrightarrow} H_{k-1} \left(\left(\varphi_{+}\right)^{\varsigma_2},\left(\varphi_{+}\right)^{\varsigma_1} \right) \longrightarrow \ldots,
	\end{split}
    \end{align}
    where $i_*$ denotes the group homomorphism induced by the inclusion mapping $i:\left(\varphi_{+}\right)^{\varsigma_1} \to \left(\varphi_{+} \right)^{\varsigma_2}$
    and $\partial_*$ stands for the boundary homomorphism.      
    Recall that $K_{\ph_+}=\{0,u_0\}$ and thanks to \eqref{32} as well as Proposition \ref{proposition_phi_plus_minus__infinity} it follows
    \begin{align}\label{34}
	H_k \left(W^{1,p}_0(\Omega),\left(\varphi_{+}\right)^{\varsigma_1} \right)=C_k\left(\varphi_{+},\infty\right)=0 \quad \text{for all }k \geq 0.
    \end{align}
    Furthermore, from Chang \cite[p. 338]{2005-Chang}, \eqref{32}, and Proposition \ref{proposition_phi_phi_plus_minus_zero} we have
    \begin{align}\label{35}
	H_k \left(W^{1,p}_0(\Omega),\left(\varphi_{+}\right)^{\varsigma_2} \right)=C_k\left(\varphi_{+},u_0\right) \quad \text{for all }k \geq 0
    \end{align}
    and
    \begin{align}\label{36}
	H_{k-1} \left(\left(\varphi_{+}\right)^{\varsigma_2},\left(\varphi_{+}\right)^{\varsigma_1} \right)=C_{k-1}\left(\varphi_{+},0\right)=\delta_{k,1}\Z \quad \text{for all }k \geq 0.
    \end{align}
    Taking into account (\ref{34}) and (\ref{36}) one observes that only the tail $k=1$ in (\ref{33}) is nontrivial. Applying the rank theorem yields
    \begin{align*}
	\rank H_1\left(\W1p0, \left(\ph_+\right)^{\varsigma_2} \right)=\rank (\ker \partial_*)+\rank (\im \partial_*).
    \end{align*}
    Then from (\ref{33})--(\ref{36}) it follows
    \begin{align}\label{40}
	\begin{split}
	  \rank C_1\left(\varphi_{+},u_0\right)
	  & =\rank \left(\ker \partial_* \right)+\rank \left(\im \partial_*\right)\\
	  & =\rank \left(\im i_*\right)+\rank \left(\im \partial_*\right)\\
	  & \leq 0+1.
	\end{split}
    \end{align}
    However, the proof of Proposition \ref{proposition_constant_sign_solutions} has shown that $u_0 \in \ints \left(C^1_0(\overline{\Omega})_+\right)$ is a critical point of $\ph_+$ of mountain pass type. Thus,
    \begin{align}\label{41}
	C_1\left(\varphi_{+},u_0\right) \neq 0.
    \end{align}
    Combining \eqref{40} and \eqref{41} yields
    \begin{align*}
	C_k\left(\varphi_{+},u_0\right)=\delta_{k,1}\Z \quad \text{for all } k\geq 0.
    \end{align*}
\end{proof}

With the aid of Proposition \ref{proposition_phi_plus_minus_u0_v0} we are now in the position to compute the critical groups of $\varphi$ at $u_0$ and $v_0$.

\begin{proposition}\label{proposition_phi_u0_v0}
    Assume H(a)$_1$ and H(f)$_1$, then
    \begin{align*}
	C_k(\varphi,u_0)=C_k(\varphi,v_0)=\delta_{k,1}\Z \quad \text{for all }k \geq 0.
    \end{align*}
\end{proposition}

\begin{proof}
    As before, we only compute $C_k(\varphi,u_0)$, the other one works similarly.
    We consider the homotopy $h: [0,1] \times W^{1,p}_0(\Omega) \to W^{1,p}_0(\Omega)$ defined by
    \begin{align*}
	h(t,u)=t\varphi(u)+(1-t)\varphi_+(u).
    \end{align*}
    Recall that $K_\ph=\{0,u_0,v_0\}$. We are going to prove the existence of a number $\rho>0$ such that $u_0$ is the only critical point of $h(t, \cdot)$ in $$B_\rho=\l\{u \in \W1p0:\l\|u-u_0\r\|_{\W1p0}<\rho \r\}$$ for all $t \in [0,1]$. Let us prove this assertion via contradiction and suppose it is not valid. Then we find a sequence $(t_n,u_n)_{n\geq 1} \subseteq [0,1] \times \W1p0$ such that
    \begin{align}\label{43}
	t_n \to t, \text{ in } [0,1], \text{ } u_n \to u_0 \text{ in } W^{1,p}_0(\Omega), \text{ and } h'_u(t_n,u_n)=0 \text{ for all }n\geq 1.
    \end{align}
    Relation (\ref{43}) gives
    \begin{align*}
	\left \lan A( u_n),v\right \ran=t_n\into  f(x,u_n) vdx+(1-t_n)\into f_+(x,u_n) v dx \quad \text{for all }v \in W^{1,p}_0(\Omega),
    \end{align*}
    which means that $u_n$ solves the problem
    \begin{equation}\label{44}
    \begin{aligned}
      -\divergenz a(\nabla u_n(x)) & = t_n f(x,u_n(x))+(1-t_n)f_+(x,u_n(x)) \quad && \text{in } \Omega,\\
       u & = 0  &&\text{on } \partial \Omega.
    \end{aligned}
    \end{equation}
    Because of \eqref{43}, from Ladyzhenskaya-Ural{\cprime}tseva \cite[p. 286]{1968-Ladyzhenskaya-Ural'tseva}, there exists  $M_{17}>0$ such that $\|u_n\|_{\Linf} \leq M_{17}$ for all $n\geq 1$ and due to Lieberman \cite[p. 320]{1991-Lieberman} we find $\beta \in (0,1)$ and $M_{18}>0$ such that
    \begin{align*}
	  \|u_n\|_{C^{1,\alpha}_0(\overline{\Omega})} \leq M_{18} \quad \text{for all }n \geq 1.
    \end{align*}
    Due to the compact embedding $C^{1,\alpha}_0(\overline{\Omega}) \hookrightarrow C^1_0(\overline{\Omega})$, we may assume that $u_n \to u_0$ in $C^1_0(\overline{\Omega})$ for a subsequence if necessary. Recalling $u_0\in \ints \left(C^1_0(\overline{\Omega})_+\right)$ there exists a number $n_0 \geq 1$ such that $\l(u_n\r)_{n \geq n_0} \subseteq \ints \left(C^1_0(\overline{\Omega})_+\right)$. Thus (\ref{44}) reduces to
    \begin{equation*}
    \begin{aligned}
      -\divergenz a(\nabla u_n(x)) & = f(x,u_n(x)) \quad && \text{in } \Omega,\\
       u & = 0  &&\text{on } \partial \Omega.
    \end{aligned}
    \end{equation*}
    Hence, $\l(u_n\r)_{n \geq n_0}$ is a sequence of distinct solutions of \eqref{problem} which contradicts the fact that $K_\ph=\{0,u_0,v_0\}$.
    
    Therefore, we find a number $\rho>0$ such that $h_u'(t, u) \neq  0$ for all $t \in [0, 1]$ and all $u\in B_\rho(u_0) \setminus\{u_0\}$. Similar to the proof of Proposition \ref{proposition_C_condition} one could verify that $h(t,\cdot)$ fulfills the C-condition for every $t \in [0, 1]$. Thus, we can invoke the homotopy invariance of critical groups to get
    \begin{align*}
	C_k(h(0,\cdot),u_0)=C_k(h(1,\cdot),u_0) \quad \text{for all }k \geq 0,
    \end{align*}
    which is equivalent to
    \begin{align*}
	C_k(\ph,u_0)=C_k(\ph_+,u_0) \quad \text{for all }k \geq 0.
    \end{align*}
    Combining this with Proposition \ref{proposition_phi_plus_minus_u0_v0} implies that
    \begin{align*}
	C_k(\varphi,u_0)=C_k(\varphi_+,u_0)=\delta_{k,1}\Z \quad \text{for all }k \geq 0.
    \end{align*}
    Similarly, we show that
    \begin{align*}
	C_k(\varphi,v_0)=\delta_{k,1}\Z \quad \text{for all }k \geq 0.
    \end{align*}
\end{proof}

Now we are ready to produce a third nontrivial solution of problem \eqref{problem}. We have the subsequent multiplicity theorem.

\begin{theorem}\label{theorem_sec3}
    Under hypotheses H(a)$_1$ and H(f)$_1$ problem (\ref{problem}) has at least three nontrivial solutions
    \begin{align*}
	u_0 \in \ints \left(C^1_0(\overline{\Omega})_+\right), \quad v_0 \in -\ints \left(C^1_0(\overline{\Omega})_+\right) \quad \text{and} \quad y_0 \in C^1_0(\overline{\Omega}).
    \end{align*}
\end{theorem}

\begin{proof}
    The existence of the two constant-sign solutions of (\ref{problem}) follows directly from Proposition \ref{proposition_constant_sign_solutions}, that is
    \begin{align*}
	u_0 \in \ints \left(C^1_0(\overline{\Omega})_+\right), \quad v_0 \in -\ints \left(C^1_0(\overline{\Omega})_+\right).
    \end{align*}
    Suppose that $K_\ph=\{0,u_0,v_0\}$ and recall that 
    \begin{align}\label{relation_u0-v0}
	C_k(\ph,u_0)=C_k(\ph,v_0)=\delta_{k,1} \Z \quad \text{for all } k \geq 0
    \end{align}
    (see Proposition \ref{proposition_phi_u0_v0}). Thanks to Proposition \ref{proposition_phi_phi_plus_minus_zero} we know that
    \begin{align}\label{relation_zero}
	C_k(\varphi,0)=\delta_{k,0}\Z \quad \text{for all } k \geq 0.
    \end{align}
    Finally, Proposition \ref{proposition_phi_infinity} implies
    \begin{align}\label{relation_infinity}
	C_k(\ph,\infty)=0 \quad \text{for all } k \geq 0.
    \end{align}
    Combining \eqref{relation_u0-v0}--\eqref{relation_infinity} and the Morse relation with $t=-1$ (see \eqref{morse_relation}) yields
    \begin{align*}
	2(-1)^1+(-1)^0=0,
    \end{align*}
    which is a contradiction. Thus, we can find $y_0 \in K_{\ph}\setminus\{0,u_0,v_0\}$ which means that $y_0$ is a third nontrivial solution of \eqref{problem} and as before, the nonlinear regularity theory guarantees that $y_0 \in C^1_0(\overline{\Omega})$. That finishes the proof.
\end{proof}

\begin{remark}
    The first multiplicity result (three-solutions-theorem) for superlinear elliptic equations has been proved by Wang \cite{1991-Wang}. In that work $p=2, a(\xi)=\xi$ for all $\xi\in \R^
    N$ (hence the differential operator is the Laplacian, semilinear equation) and $f(x,\cdot)=f(\cdot)$ (i.e., the reaction is $x$-independent), $f \in C^1(\R), f'(0)=0$ and it satisfies the Ambrosetti-Rabinowitz condition (see \eqref{AR1}, \eqref{AR2}). We point out that Theorem \ref{theorem_sec3} extends significantly the multiplicity result of Wang \cite{1991-Wang}. Other multiplicity results for $p$-Laplacian equations with a superlinear reaction satisfying more restrictive conditions than H(f)$_1$ were proved by Liu \cite{2010-Liu} and Sun \cite{2010-Sun}. For Neumann problems driven by the $p$-Laplacian we refer to Aizicovici-Papageorgiou-Staicu \cite{2009-Aizicovici-Papageorgiou-Staicu}.
\end{remark}

\section{Five nontrivial solutions}\label{section_four}

In this section we produce additional nontrivial solutions for problem \eqref{problem} by changing the geometry of the problem near the origin. Roughly speaking we require that $f(x,\cdot)$ exhibits an oscillatory behavior near zero. We also suppose some stronger assumptions on the map $a(\cdot)$ which allows us to prove the existence of five nontrivial solutions of \eqref{problem} given with complete sign information. The results in this section extend the recent work of Aizicovici-Papageorgiou-Staicu \cite{2012-Aizicovici-Papageorgiou-Staicu}.

The new hypotheses on the map $a(\cdot)$ are the following.

\begin{enumerate}[leftmargin=1.2cm]
    \item[H(a)$_2$:]
	$a(\xi)=a_0(\|\xi\|)\xi$ for all $\xi \in \R^N$ with $a_0(t)>0$ for all $t>0$, hypotheses H(a)$_2$(i)--(iii) are the same as the corresponding hypotheses H(a)$_1$(i)--(iii) and
	\begin{enumerate}
	    \item[(iv)]
		$pG_0(t)-t^2a_0(t) \geq - c_6$ for all $t>0$ and some $c_6>0$;
	    \item[(v)]
		there exists $q \in (1,p)$ such that $t \mapsto G_0\l(t^{\frac{1}{q}}\r)$ is convex in $(0,+\infty)$, 
		\begin{align*}
		    \limsup_{t \to 0^+} \frac{qG_0(t)}{t^q}<+\infty,
		\end{align*}    
		and $t^2a_0(t)-qG_0(t) \geq \hat{\eta} t^p$ for all $t>0$ and some $\hat{\eta}>0$.
	\end{enumerate}
\end{enumerate}

\begin{remark}
    The examples given in Example \ref{example_a} still satisfy the new hypotheses H(a)$_2$. Note that hypothesis H(a)$_2$(v) implies
    \begin{align}\label{new_estimate}
	G(\xi) \leq c_7 (\|\xi\|^q+\|\xi\|^p) \quad \text{for all }\xi \in \R^N,
    \end{align}
    with some $c_7>0$.
\end{remark}

Furthermore, we suppose new hypotheses on the reaction $f: \Omega \times \R \to \R$ as follows.

\begin{enumerate}[leftmargin=1.2cm]
    \item[H(f)$_2$:]
	$f: \Omega \times \R \to \R$ is a Carath\'{e}odory function such that $f(x,0)=0$ for a.a. $x\in\Omega$, hypotheses H(f)$_2$(i)-(iii) are the same as the corresponding hypotheses H(f)$_1$(i)--(iii) and
	\begin{enumerate}
	    \item[(iv)]
		there exist $\zeta \in (1,q)$ ($q$ as in hypothesis H(a)$_2$(v)) and $\delta>0$ such that
		\begin{align*}
		    \zeta F(x,s) \geq f(x,s)s > 0 \quad \text{for a.a. }x\in \Omega \text{ and for all }0<|s| \leq \delta
		\end{align*}
		and
		\begin{align*}
		    \essinf_\Omega F(\cdot,\pm\delta)>0;
		\end{align*}
	    \item[(v)]
		there exist real numbers $\xi_-<0<\xi_+$ such that
		\begin{align*}
		    f(x,\xi_+)\leq \eta_1<0<\eta_2 \leq f(x,\xi_-) \quad \text{for a.a. }x\in \Omega;
		\end{align*}
	    \item[(vi)]
		for every $\varrho>0$, there exists $\xi_\varrho>0$ such that
		\begin{align*}
		    s\mapsto f(x,s)+\xi_\varrho|s|^{p-2}s
		\end{align*}
		is nondecreasing on $[-\varrho,\varrho]$ for a.a. $x\in\Omega$.
	\end{enumerate}
\end{enumerate}

\begin{remark}
    Hypothesis H(f)$_2$(iv) implies that $F(x,s)\geq M_{19}|s|^\zeta$ for a.a. $x\in \Omega$, for all $|s|\leq \delta$, and some $M_{19}>0$. We also point out that $f(x,\cdot)$ exhibits an oscillatory behavior near zero which follows directly from hypothesis H(f)$_2$(v).
\end{remark}

\begin{example}
    As before, we drop the $x$-dependence. The following function satisfies hypotheses H(f)$_2$.
    \begin{align*}
	f(s)=
	\begin{cases}
	    |s|^{\tau-2}s-2|s|^{p-2}s \quad & \text{if } |s|\leq 1,\\
	    |s|^{p-2}s\ln |s| - |s|^{q-2}s & \text{if }|s|>1
	\end{cases} \quad \text{with }1<q,\tau<p.
    \end{align*}
    Note that this $f$ does not satisfy the Ambrosetti-Rabinowitz condition.
\end{example}

First we produce two nontrivial constant sign solutions.

\begin{proposition}\label{proposition_Sec4_two_constant_sign}
    Let the hypotheses H(a)$_2$ and H(f)$_2$ be satisfied. Then problem (\ref{problem}) has at least two nontrivial constant sign solutions $u_0 \in \ints \left(C^1_0(\overline{\Omega}) \right)$ and $v_0 \in -\ints \left(C^1_0(\overline{\Omega}) \right)$ such that
    \begin{align*}
	\xi_- < v_0(x) \leq 0 \leq u_0(x) < \xi_+ \quad \text{for all } x \in \close.
    \end{align*}
    Moreover, both solutions are local minimizers of the energy functional $\varphi$.
\end{proposition}

\begin{proof}
    Let $\hat{f}_+ : \Omega \times \R \to \R$ be the truncation function defined by
    \begin{align}\label{Sec4_3}
	\begin{split}
	    \hat{f}_+(x,s)=
	    \begin{cases}
		0 \quad & \text{if } s<0\\
		f(x,s) & \text{if } 0 \leq s \leq \xi_+\\
		f(x,\xi_+) & \text{if } \xi_+<s
	    \end{cases},
	\end{split}
    \end{align}
    which is known to be a Carath\'{e}odory function. We introduce the $C^1$-functional $\hat{\varphi}_+ : W^{1,p}_0(\Omega) \to \R$ through
    \begin{align*}
	\hat{\varphi}_+(u)=\into G(\nabla u(x))dx - \into \hat{F}_+(x,u(x))dx
    \end{align*}
    with $\hat{F}_+(x,s)=\int^s_0 \hat{f}_+(x,t)dt$. It is clear that $\hat{\varphi}_+: W^{1,p}_0(\Omega) \to \R$ is coercive (see Corollary \ref{corollary_upper_lower_estimates}, \eqref{Sec4_3}) and sequentially weakly lower semicontinuous. Hence, its global minimizer $u_0 \in W^{1,p}_0(\Omega)$ exists, that is
    \begin{align*}
	\hat{\varphi}_+(u_0)=\inf \left\{\hat{\varphi}_+(u): u \in W^{1,p}_0(\Omega) \right \}=\hat{m}_+.
    \end{align*}
    By virtue of hypothesis H(f)$_2$(v) we know that we can find $\beta>0$ and $\delta_0 \in \left(0,\min \left\{\delta,\xi_+ \right\} \right)$ such that
    \begin{align}\label{Sec4_4}
	G(\xi) \leq \beta \|\xi\|^q \quad \text{for all } \|\xi\|\leq \delta_0.
    \end{align}
    Recall that hypothesis H(f)$_2$(iv) implies
    \begin{align}\label{Sec4_1}
	F(x,s) \geq M_{20} |s|^\zeta \quad \text{for a.a. }x \in \Omega \text{ and for all }|s| \leq \delta_0,
    \end{align}
    with some $M_{20}>0$. Since $\hat{u}_1(q) \in \interior$ we can choose $t \in (0,1)$ sufficiently small such that $t \hat{u}_1(q)(x) \in [0,\delta_0]$ for all $x \in \close$. Taking into account \eqref{Sec4_4}, \eqref{Sec4_1} and $\|\hat{u}_1(q)\|_{\Lp{q}}=1$, we obtain
    \begin{align}\label{Sec4_2}
      \begin{split}
	\hat{\varphi}_+(t \hat{u}_1(q))
	& =\into G\l(\nabla t \hat{u}_1\r)dx- \into \hat{F}_+(x,t \hat{u}_1) dx\\
	& \leq \beta t^q \l\|\nabla(\hat{u}_1(q)) \r\|_{L^q(\Omega)}^q-M_{20}t^\zeta \l\|\hat{u}_1(q)\r\|^\zeta_{L^\zeta(\Omega)}\\
	& = \beta t^q \hat{\lambda}_1(q)-M_{20}t^\zeta \l\|\hat{u}_1(q)\r\|^\zeta_{L^\zeta(\Omega)}.
      \end{split}
    \end{align}
    Since $\zeta<q$, choosing $t \in (0,1)$ small enough, \eqref{Sec4_2} gives
    \begin{align*}
	\hat{\varphi}_+(t \hat{u}_1(q))<0,
    \end{align*}
    meaning
    \begin{align*}
	\hat{\varphi}_+(u_0)=\hat{m}_+<0=\hat{\varphi}_+(0).
    \end{align*}
    We conclude
    \begin{align}\label{Sec4_5}
	u_0 \neq 0.
    \end{align}
    On the other hand, since $u_0$ is a critical point of $\hat{\varphi}_+$ there holds
    \begin{align}\label{Sec4_6}
	\left \lan A u_0,v \right\ran = \left\lan N_{\hat{f}_+}(u_0),v\right \ran\quad \text{for all }v \in W^{1,p}_0(\Omega).
    \end{align}
    Choosing $v=-u_0^-$ as test function in (\ref{Sec4_6}) and applying Lemma \ref{lemma_properties}(c) as well as the definition of the truncation (see \eqref{Sec4_3}) yields
    \begin{align*}
	\frac{c_1}{p-1} \l\|\nabla u_0^-\r\|^p_{L^p(\Omega)} \leq 0.
    \end{align*}
    Hence,
    \begin{align}\label{Sec4_9}
	u_0 \geq 0.
    \end{align}
    Now, making use of hypothesis H(f)$_2$(v) and taking $(u_0-\xi_+)^+ \in W^{1,p}_0(\Omega)$ as test function in \eqref{Sec4_6} one gets
    \begin{align}\label{Sec4_7}
	\begin{split}
	    \into \left(a(\nabla u_0),\nabla \left(u_0-\xi_+\right)^+ \right)_{\R^N}dx
	    & = \into \hat{f}_+(x,u_0)\left(u_0-\xi_+\right)^+dx\\
	    & = \into f(x,\xi_+)\left(u_0-\xi_+\right)^+dx\\
	    & \leq 0.
	\end{split}
    \end{align}
    From \eqref{Sec4_7} it follows
    \begin{align*}
	\int_{\left\{u_0>\xi_+ \right\}} \left(a(\nabla u_0)-a(\nabla \xi_+),\nabla u_0-\nabla \xi_+ \right)_{\R^N}dx\leq 0,
    \end{align*}
    and by virtue of Lemma \ref{lemma_properties}(a),
    \begin{align*}
	\left|\left\{u_0>\xi_+\right\} \right |_N =0.
    \end{align*}
    Hence,
    \begin{align}\label{Sec4_8}
	u_0(x) \leq \xi_+ \quad \text{a.e. in } \Omega.
    \end{align}
    Combining \eqref{Sec4_5}, \eqref{Sec4_9} and \eqref{Sec4_8} we have
    \begin{align*}
	0 \leq u_0(x) \leq \xi_+ \text{ a.e. in } \Omega \text{ and } u_0 \neq 0.
    \end{align*}
    Then, \eqref{Sec4_6} becomes
    \begin{align*}
	\left \lan A u_0,v \right\ran = \left\lan N_{f}(u_0),v\right \ran\quad \text{for all }v \in W^{1,p}_0(\Omega),
    \end{align*}
    meaning that
    \begin{equation*}
	\begin{aligned}
	    -\divergenz a(\nabla u_0(x)) & = f(x,u_0(x)) \quad && \text{in } \Omega,\\
	    u & = 0  &&\text{on } \partial \Omega.
	\end{aligned}
    \end{equation*}
    The nonlinear regularity theory ensures that $u_0 \in C^1_0(\close)$ (see Ladyzhenskaya-Ural{\cprime}tseva \cite{1968-Ladyzhenskaya-Ural'tseva} and Lieberman \cite[p. 320]{1991-Lieberman}).

    Thanks to hypothesis H(f)$_2$(vi) we find for $\varrho=\xi_+$ a constant $\xi_\varrho >0$ such that
    \begin{align*}
	-\divergenz a(\nabla u_0(x))+\xi_\varrho u_0(x)^{p-1}=f(x,u_0(x))+\xi_\varrho u_0(x)^{p-1} \geq 0 \quad \text{for a.a. }x\in \Omega.
    \end{align*}
    Hence,
    \begin{align*}
	\divergenz a(\nabla u_0(x))\leq \xi_\varrho u_0(x)^{p-1} \quad \text{for a.a. }x\in \Omega.
    \end{align*}
    Due to Hypothesis H(a)$_2$(iv) the strong maximum principle implies that $u_0 \in \ints \left(C^1_0(\overline{\Omega})_+\right)$ (see Pucci-Serrin \cite[pp. 111 and 120]{2007-Pucci-Serrin}).

    Now, let $\delta>0$ and set $u_\delta=u_0+\delta \in C^1(\close)$. Recall that $u_0(x)\leq \xi_+$ for all $x \in \close$, by means of hypotheses H(f)$_2$(v), (vi), we have
    \begin{align}\label{Sec4_12}
	\begin{split}
	    -\divergenz a(\nabla u_\delta(x))+\xi_\varrho u_\delta(x)^{p-1}
 	    & \leq -\divergenz a(\nabla u_0(x))+\xi_\varrho u_0(x)^{p-1}+o(\delta)\\
	    & = f(x,u_0(x))+\xi_\varrho u_0(x)^{p-1} +o(\delta)\\
	    & \leq f(x, \xi_+)+\xi_\varrho \xi_+^{p-1}+o(\delta)\\
	    & \leq \eta_1+\xi_\varrho \xi_+^{p-1}+o(\delta).
	\end{split}
    \end{align}
    Recall that $\eta_1<0$ (see H(f)$_2$(v)) and $o(\delta) \to 0^+$ as $\delta\to 0^+$. Then, for $\delta>0$ sufficiently small there holds $\eta_1+o(\delta)\leq 0$. Hence, from \eqref{Sec4_12} we obtain
    \begin{align*}
	\begin{split}
	    -\divergenz a(\nabla u_\delta(x))+\xi_\varrho u_\delta(x)^{p-1}v\leq -\divergenz a(\nabla \xi_+)+\xi_\varrho \xi_+^{p-1}.
	\end{split}
    \end{align*}
    Applying again Pucci-Serrin \cite[p. 61]{2007-Pucci-Serrin} it follows
    \begin{align*}
	u_\delta(x) \leq \xi_+ \quad \text{for all }x\in \Omega,
    \end{align*}
    consequently,
    \begin{align*}
	u(x) < \xi_+ \quad \text{for all }x\in \close.
    \end{align*}
    Therefore, we have
    \begin{align*}
	u_0 \in \ints_{C^1_0(\close)}[0,\xi_+].
    \end{align*}
    Since $\varphi \big|_{[0,\xi_+]}=\hat{\varphi}_+\big|_{[0,\xi_+]}$ we conclude that $u_0$ is a local $C^1_0(\close)$-minimizer of $\varphi$. So, Proposition \ref{proposition_local_minimizers} implies that $u_0$ is a local $W^{1,p}_0(\Omega)$-minimizer of $\varphi$.
    
    For the nontrivial negative solution we introduce the following truncation of the reaction $f(x,\cdot)$
    \begin{align*}
	\begin{split}
	    \hat{f}_-(x,s)=
	    \begin{cases}
		f(x,\xi_-) \quad & \text{if } s<\xi_-\\
		f(x,s) & \text{if } \xi_- \leq s \leq 0\\
		0 & \text{if } 0<s
	    \end{cases},
	\end{split}
    \end{align*}
    which is a Carath\'{e}odory function. Setting $\hat{F}_-(x,s)=\int_0^s \hat{f}_-(x,t)dt$ we consider the  $C^1$-functional $\hat{\ph}_-: \W1p0 \to \R$ defined by
    \begin{align*}
	\hat{\varphi}_-(u)=\into G(\nabla u(x))dx - \into \hat{F}_-(x,u(x))dx.
    \end{align*}
    Working as above via the direct method we produce a solution $v_0 \in - \interior$ being a local minimizer of $\ph$.
\end{proof}

\begin{remark}
    A careful inspection of the proof above reveals that we only needed hypotheses H(f)$_2$(iv), (v), (vi), i.e., the asymptotic conditions at $\pm \infty$ (see H(f)$_2$(ii), (iii)) are irrelevant. Moreover, the global growth condition H(f)$_2$(i) can be replaced by the following local one.
    \begin{center}
	For every $\varrho>0$ there exists $a_\varrho \in \Linf_+$ such that
	\begin{align*}
	    |f(x,s)| \leq a_\varrho(x) \quad \text{for a.a. }x\in \Omega \text{ and for all }|s|\leq \varrho.
	\end{align*}
    \end{center}

\end{remark}

Using these two nontrivial constant sign solutions we can produce two more precisely localized with respect to $u_0$ and $v_0$. Now we need the asymptotic conditions at $\pm \infty$.

\begin{proposition}\label{proposition_Sec4_four_constant_sign}
    Under the hypotheses H(a)$_2$ and H(f)$_2$ problem \eqref{problem} possesses two more nontrivial constant sign solutions $u_1 \in \interior$ and $v_1 \in -\interior$ satisfying
    \begin{align*}
	u_0(x)\leq u_1(x) \quad \text{and} \quad v_1(x)\leq v_0(x) \quad \text{for all }x\in \close
    \end{align*}
    with $u_1\neq u_0$ and $v_1\neq v_0$.
\end{proposition}

\begin{proof}
    We begin with the proof for the existence of $u_1$. For $u_0 \in \interior$ being the constant sign solution obtained in Proposition \ref{proposition_Sec4_two_constant_sign} we define the truncation mapping $e_+: \Omega \times \R \to \R$ through
    \begin{align}\label{Sec4_14}
	e_+(x,s)=
	\begin{cases}
	    f(x,u_0(x)) \quad & \text{if } s<u_0(x),\\
	    f(x,s) & \text{if } u_0(x) \leq s,
	\end{cases}
    \end{align}
    which is again a Carath\'{e}odory function. Setting $E_+(x,s)=\int_0^s e_+(x,t)dt$ we introduce the $C^1$-functional $\sigma_+ : W^{1,p}_0(\Omega) \to \R$ by
    \begin{align*}
	\sigma_+(u)=\into G(\nabla u(x))dx-\into E_+(x,u(x))dx.
    \end{align*}
    First we note that $\sigma_+$ fulfills the $C$-condition which can be shown as in the proof of Proposition \ref{proposition_C_condition} with minor modifications by applying \eqref{Sec4_14}.
    
    {\bf Claim:} We may assume that $u_0 \in \interior$ is a local minimizer of the functional $\sigma_+$.
    
    Recalling $u_0(x)< \xi_+$ for all $x\in \close$ we introduce the subsequent Carath\'{e}odory truncation function
     \begin{align}\label{Sec4_15}
	\hat{e}_+(x,s)=
	\begin{cases}
	    e_+(x,s) \quad & \text{if } s \leq \xi_+\\
	    e_+(x,\xi_+) & \text{if } s > \xi_+
	\end{cases}
    \end{align}
    and consider the $C^1$-functional $\hat{\sigma}_+:W^{1,p}_0(\Omega) \to \R$
    \begin{align*}
	\hat{\sigma}_+(u)=\into G(\nabla u(x)) dx -\into \hat{E}_+(x,u(x))dx
    \end{align*}
    with $\hat{E}_+(x,s)=\int^s_0 \hat{e}_+(x,t)dt$. Obviously, $\hat{\sigma}_+$ is coercive and sequentially weakly lower semicontinuous which implies due to the Weierstrass theorem that there is a global minimizer $\hat{u}_0 \in W^{1,p}_0(\Omega)$ meaning
    \begin{align*}
	\hat{\sigma}_+(\hat{u}_0)=\inf \left\{\hat{\sigma}_+(u): u \in W^{1,p}_0(\Omega) \right\}.
    \end{align*}
    In particular, this gives $\hat{\sigma}_+'(\hat{u}_0)=0$ and hence,
    \begin{align}\label{Sec4_16}
	\left\lan A(\hat{u}_0),v\right \ran=\left \lan N_{\hat{e}_+}(\hat{u}_0),v\right\ran \quad \text{for all }v \in W^{1,p}_0(\Omega).
    \end{align}
    Taking $v=\left(u_0-\hat{u}_0 \right)^+ \in \W1p0$ in the last equation and using \eqref{Sec4_14}, \eqref{Sec4_15} we obtain
    \begin{align*}
	\left\lan A(\hat{u}_0),\left(u_0-\hat{u}_0 \right)^+\right \ran
	& =\into \hat{e}_+(x,\hat{u}_0))\left(u_0-\hat{u}_0 \right)^+dx\\
	& = \into  f(x,u_0) \left(u_0-\hat{u}_0 \right)^+ dx\\
	& = \left\lan A(u_0),\left(u_0-\hat{u}_0 \right)^+\right \ran.
    \end{align*}
    It follows that
    \begin{align*}
	\left\lan A(u_0)-A(\hat{u}_0),\left(u_0-\hat{u}_0 \right)^+\right \ran=0,
    \end{align*}
    meaning
    \begin{align*}
	\int_{\{u_0>\hat{u}_0\}} \left(a(\nabla u_0)-a(\nabla \hat{u}_0),\nabla u_0-\nabla \hat{u}_0 \right)_{\RN}dx=0.
    \end{align*}
    Hence, $|\{u_0>\hat{u}_0\}|_N=0$, that is, $u_0 \leq \hat{u}_0$. Now, taking $v=\left(\hat{u}_0-\xi_+ \right)^+$ in \eqref{Sec4_16}, applying \eqref{Sec4_14}, \eqref{Sec4_15}, H(f)$_2$(v), and recalling $u_0(x)<\xi_+$ for all $x \in \Omega$, we get
    \begin{align*}
	\left\lan A(\hat{u}_0),\left(\hat{u}_0-\xi_+ \right)^+\right \ran
	& =\into \hat{e}_+(x,\hat{u}_0))\left(\hat{u}_0-\xi_+ \right)^+dx\\
	& = \into  f(x,\xi_+) \left(\hat{u}_0-\xi_+ \right)^+ dx\\
	& \leq 0,
    \end{align*}
    which implies
    \begin{align*}
	\int_{\{\hat{u}_0>\xi_+\}} \l\|\nabla \hat{u}_0\r\|^p dx \leq 0
    \end{align*}
    (see Lemma \ref{lemma_properties}(c)). As above we conclude that $\l|\l\{\hat{u}_0>\xi_+\r\}\r|_N=0$, i.e., $\hat{u}_0\leq \xi_+$. Then, $\hat{u}_0 \in [u_0,\xi_+]$ and  equation \eqref{Sec4_16} becomes
    \begin{align*}
	\left\lan A(\hat{u}_0),v\right \ran=\left \lan N_{f}(\hat{u}_0),v\right\ran \quad \text{for all }v \in W^{1,p}_0(\Omega),
    \end{align*}
    which means that $\hat{u}_0$ solves our original problem \eqref{problem}. Applying again the nonlinear regularity theory we obtain that $\hat{u}_0 \in \interior$ (see the proof of Proposition \ref{proposition_Sec4_two_constant_sign}).
    If $\hat{u}_0 \neq u_0$, then the assertion of the proposition is proved and we are done.

    Let us suppose that $\hat{u}_0=u_0$. By means of the truncations in (\ref{Sec4_14}),\eqref{Sec4_15} we have
    \begin{align*}
	\sigma_+\big|_{[0,\xi_+]}=\hat{\sigma}_+\big|_{[0,\xi_+]}.
    \end{align*}
    Since $\hat{u}_0=u_0 \in \ints_{C^1_0(\close)}[0,\xi_+]$ we see that $\hat{u}_0=u_0$ is a local $C_0^1(\close)$-minimizer of $\sigma_+$ and with regard to Proposition \ref{proposition_local_minimizers} it is also a local $\W1p0$-minimizer of $\sigma_+$. 
    This proves the claim.

    We may also assume that $u_0$ is an isolated critical point of $\sigma_+$, otherwise we would find a sequence $(u_n)_{n \geq 1} \subseteq \W1p0$ such that
    \begin{align}\label{equation_2}
	u_n \to u_0 \text{ in } \W1p0 \quad \text{and} \quad \sigma_+'(u_n)=0 \quad \text{for all } n\geq 1.
    \end{align}
    It follows
    \begin{align*}
	A(u_n)=N_{e_+}(u_n) \quad \text{for all }n \geq 1
    \end{align*}
    meaning that
    \begin{align}\label{equation_1}
	-\divergenz a(\nabla u_n(x))=e_+(x,u_n(x)) \quad \text{a.e. in }\Omega.
    \end{align}
    Then, from \eqref{equation_2}, \eqref{equation_1} and Ladyzhenskaya-Ural{\cprime}tseva \cite{1968-Ladyzhenskaya-Ural'tseva} we can find $M_{21}>0$ such that $\|u_n\|_{\Linf} \leq M_{21}$. Applying the regularity results of Lieberman \cite{1991-Lieberman} we find $\gamma \in (0,1)$ and $M_{22}>0$ such that
    \begin{align*}
	u_n \in C_0^{1,\gamma}(\close) \text{ and } \|u_n\|_{C_0^{1,\gamma}(\close)} \leq M_{22} \quad \text{for all }n \geq 1.
    \end{align*}
    Exploiting the compact embedding of $C^{1,\gamma}(\close)$ into $C^1_0(\close)$ and by virtue of \eqref{equation_2} one gets
    \begin{align*}
	u_n \to u_0, \quad u_n \geq u_0 \quad \text{for all }n \geq 1.
    \end{align*}
    That means we have proved the existence of a whole sequence $(u_n)_{n\geq 1} \subseteq \interior$ of distinct nontrivial positive solutions of \eqref{problem}. Hence, we are done. Therefore, we may consider $u_0$ as an isolated critical point of $\sigma_+$.

    Because of the claim there exists a number $\varrho \in (0,1)$ such that
    \begin{align}\label{Sec4_18}
	\sigma_+(u_0)<\inf \left\{\sigma_+(u):\|u-u_0\|_{\W1p0}=\varrho \right\}=:\eta_\varrho^+
    \end{align}
    (see Aizicovici-Papageorgiou-Staicu \cite[Proof of Proposition 29]{2008-Aizicovici-Papageorgiou-Staicu}). Recall that $\sigma_+$ satisfies the $C$-condition. Thanks to hypothesis H(f)$_2$(ii) we verify that if $u \in \interior$, then $\sigma_+(tu) \to -\infty$ as $t \to +\infty$. These facts combined with \eqref{Sec4_18} permit the usage of the mountain pass theorem stated in Theorem \ref{theorem_mountain_pass}. This provides the existence of $u_1 \in W^{1,p}_0(\Omega)$ such that
    \begin{align}\label{Sec4_20}
	u_1 \in K_{\sigma_+} \quad \text{and} \quad \eta^+_\varrho \leq e_+(u_1).
    \end{align}
    With a view to \eqref{Sec4_18} and \eqref{Sec4_20} we see that $u_0 \leq u_1, u_0 \neq u_1$ and $u_1 \in \interior$ solves problem \eqref{problem}.
    
    The case of a second nontrivial negative solution $v_1 \in -\interior$ with $v_1 \leq v_0$ and $v_1 \neq v_0$ can be shown using similar arguments.
\end{proof}

Now we are interested to find a fifth solution of (\ref{problem}) being a sign-changing one. In order to produce the nodal solution we will use some tools from Morse theory. For this purpose we start by computing the critical groups at the origin of the $C^1$-energy functional $\ph:\W1p0 \to \R^N$ defined by
\begin{align*}
    \ph(u)=\into G(\nabla u(x))dx-\into F(x,u(x))dx.
\end{align*}
Our proof uses ideas from Moroz \cite{1997-Moroz} in which $G(\xi)=\frac{1}{2}\|\xi\|^2$ for all $\xi \in \R^N$ with more restrictive conditions on $f:\Omega \times \R \to \R$ and from Jiu-Su \cite{2003-Jiu-Su} where $G(\xi)=\frac{1}{p}\|\xi\|^p$ for all $\xi \in \R^N$.

\begin{proposition}\label{proposition_Sec4_phi_zero}
    Under the assumptions H(a)$_2$ and H(f)$_2$(i),(iv) there holds $C_k(\varphi,0)=0$ for all $k \geq 0$.
\end{proposition}

\begin{proof}
    Note that from H(f)$_2$(i) and (iv) we have
    \begin{align}\label{Sec4_22}
	F(x,s) \geq M_{23} |s|^\zeta-M_{24} |s|^{r} \quad \text{for a.a. }x \in \Omega \text{ and for all }s\in \R
    \end{align}
    with positive constants $M_{23}, M_{24}$. Recall that hypothesis H(a)$_2$(v) implies
    \begin{align}\label{Sec4_400}
	G(\xi) \leq c_7 (\|\xi\|^q+\|\xi\|^p) \quad \text{for all } \xi \in \R^N
    \end{align}
    (see also \eqref{new_estimate}). Let $u \in W^{1,p}_0(\Omega)$ and $t>0$. Combining \eqref{Sec4_22} and \eqref{Sec4_400} gives
    \begin{align*}
	\begin{split}
	    \varphi(tu)
	    & = \into G(\nabla (tu))dx-\into F(x,tu)dx\\
	    & \leq c_7 t^q \|\nabla u\|^q_{L^q(\Omega)}+c_7 t^p\|\nabla u\|_{L^p(\Omega)}^p-M_{23} t^\zeta \|u\|_{L^\zeta(\Omega)}^\zeta+M_{24} t^{r} \|u\|_{L^{r}(\Omega)}^{r}.
	\end{split}
    \end{align*}
    Since $\zeta<q< p <r$ there exists a small number $t_0>0$ such that
    \begin{align*}
	\varphi(tu)<0 \quad \text{for all } 0<t<t_0.
    \end{align*}
    Now let $u \in \W1p0$ be such that $\varphi(u)=0$. Taking into account H(a)$_2$(v), H(f)$_2$(i), (iv), and the Sobolev embedding theorem it follows
    \begin{align}\label{unclear}
      \begin{split}
	\frac{d}{dt}\varphi(tu)\bigg|_{t=1}
	& = \lan \varphi'(tu),u\ran\bigg|_{t=1}\\
	& = \into \left(a(\nabla u),\nabla u \right)_{\RN}dx -\into f(x,u)udx\\
	& \qquad -\zeta\into G(\nabla u)dx+\into \zeta F(x,u)dx\\
	& \geq \hat{\eta} \|\nabla u\|^p_{\Lp{p}}+ \into \left[\zeta F(x,u)-f(x,u)u \r]dx \\
	& \geq \hat{\eta} \|u\|^p_{\W1p0}- M_{25} \|u\|^r_{\W1p0}
      \end{split}
    \end{align}
    with some $M_{25}>0$. Since $p<r$ we can find $\varrho \in (0,1)$ small enough such that
    \begin{align}\label{Sec4_25}
	\frac{d}{dt}\varphi(tu)\bigg|_{t=1}>0 \quad \forall u \in \W1p0 \text{ with } \varphi(u)=0 \text{ and } 0<\|u\|_{\W1p0}\leq \varrho.
    \end{align}
    Now, let $u \in \W1p0$ with $0<\|u\|_{\W1p0} \leq \varrho$ and $\ph(u)=0$. In the following we are going to show that
    \begin{align}\label{Sec4_60}
	\ph(tu)\leq 0 \quad \text{for all }t \in [0,1].
    \end{align}
    Arguing indirectly, suppose that we can find a number $t_0 \in (0,1)$ such that $\ph(t_0 u)>0$. Since $\ph$ is continuous and $\ph(u)=0$ there exists $t_1 \in (t_0,1]$ such that $\ph(t_1 u)=0$. Let $t_*=\min \left \{ t \in [t_0,1]: \ph(tu)=0\right\}$. It is clear that $t_*>t_0>0$ and 
    \begin{align}\label{Sec4_61}
	\ph(tu) >0 \quad \text{for all }t \in [t_0,t_*).
    \end{align}
    
    Setting $v=t_*u$ we have $0<\|v\|_{\W1p0}\leq \|u\|_{\W1p0}\leq \varrho$ and $\ph(v)=0$. Then, \eqref{Sec4_25} gives
    \begin{align}\label{Sec4_62}
	\frac{d}{dt}\varphi(tv)\bigg|_{t=1}>0.
    \end{align}
    Moreover, from \eqref{Sec4_61} we obtain
    \begin{align*}
	\ph(v)=\ph(t_* u)=0<\ph(tu) \quad \text{for all }t \in [t_0,t_*).
    \end{align*}
    Hence,
    \begin{align}\label{Sec4_63}
	\frac{d}{dt}\varphi(tv)\bigg|_{t=1}=t_* \frac{d}{dt}\varphi(tu)\bigg|_{t=t_*}=t_* \lim_{t \to t_*^-} \frac{\ph(tu)}{t-t_*} \leq 0.
    \end{align}
    Comparing \eqref{Sec4_62} and \eqref{Sec4_63} we reach a contradiction. This proves \eqref{Sec4_60}.
    
    By taking $\varrho \in (0,1)$ even smaller if necessary we may assume that $K_\ph \cap \overline{B}_\varrho=\{0\}$ where $\overline{B}_\varrho=\l\{u \in \W1p0: \|u\|_{\W1p0} \leq \varrho\r\}$. Let $h:[0,1] \times \left(\ph^0 \cap \overline{B}_\varrho \right) \to \ph^0 \cap \overline{B}_\varrho$ be the deformation defined by
    \begin{align*}
	h(t,u)=(1-t)u.
    \end{align*}
    Thanks to \eqref{Sec4_60} we verify that this deformation is well-defined and it implies that $\ph^0 \cap \overline{B}_\varrho$ is contractible in itself.
    
    Fix $u \in \overline{B}_\varrho$ with $\ph(u)>0$. We show that there exists an unique $t(u) \in (0,1)$ such that
    \begin{align*}
	\ph(t(u)u)=0.
    \end{align*}
    
    Since $\ph(u)>0$ and the continuity of $t \mapsto \ph(tu)$, \eqref{Sec4_25} ensures the existence of such a $t(u) \in (0,1)$. It remains to show its uniqueness. Arguing by contradiction, suppose that for $0<t^*_1=t(u)_1<t^*_2=t(u)_2<1$ we have $\ph(t^*_1u)=\ph(t^*_2u)=0$. Then, \eqref{Sec4_60} implies
    \begin{align*}
	\gamma(t)=\ph(tt^*_2u) \leq 0 \quad \text{for all }t \in [0,1].
    \end{align*}
    Therefore $\frac{t^*_1}{t^*_2}\in (0,1)$ is a maximizer of $\gamma$ and thus,
    \begin{align*}
	\frac{d}{dt} \gamma(t) \bigg|_{t=\frac{t^*_1}{t^*_2}}=0,
    \end{align*}
    which implies that
    \begin{align*}
	\frac{t^*_1}{t^*_2} \frac{d}{dt} \ph(tt^*_2u)\bigg|_{t=\frac{t^*_1}{t^*_2}}=\frac{d}{dt}\ph(tt^*_1u)\bigg|_{t=1}=0.
    \end{align*}
    But this is a contradiction to \eqref{Sec4_25} and the uniqueness of $t(u) \in (0,1)$ is proved.

    This uniqueness implies that
    \begin{align*}
	\ph(tu)<0 \quad \text{if }t \in (0,t(u)) \quad \text{and} \quad \ph(tu)>0 \quad \text{for all }t \in (t(u),1].
    \end{align*}
    Let $T_1: \overline{B}_\varrho \setminus \{0\} \to (0,1]$ be defined by
    \begin{align*}
	T_1(u)=
	\begin{cases}
	    1 \quad &\text{if } \ph(u) \leq 0,\\
	    t(u) &\text{if }\ph(u)>0.
	\end{cases}
    \end{align*}
    It is easy to check that $T_1$ is continuous. Next, we consider a map $T_2: \overline{B}_\varrho \setminus\{0\}\to \left(\ph^0 \cap \overline{B}_\varrho\right)\setminus\{0\}$ defined by
    \begin{align*}
	T_2(u)=
	\begin{cases}
	    u \quad &\text{if } \ph(u) \leq 0,\\
	    T_1(u)u &\text{if }\ph(u)>0.
	\end{cases}
    \end{align*}
    Obviously, $T_2$ is a continuous function. We observe that
    \begin{align*}
	T_2\bigg|_{\left(\ph^0 \cap \overline{B}_\varrho\right)\setminus\{0\}}=\id\bigg|_{\left(\ph^0 \cap \overline{B}_\varrho\right)\setminus\{0\}},
    \end{align*}
    which proves that $\left(\ph^0 \cap \overline{B}_\varrho\right)\setminus\{0\}$ is a retract of $\overline{B}_\varrho \setminus\{0\}$. Note that $\overline{B}_\varrho\setminus\{0\}$ is contractible in itself. Therefore, the same is true for $\left(\ph^0 \cap \overline{B}_\varrho\right)\setminus\{0\}$. Previously, we proved that $\ph^0 \cap \overline{B}_\varrho$ is contractible in itself. From Granas-Dugundji \cite[p. 389]{2003-Granas-Dugundji} it follows that
    \begin{align*}
	H_k\left(\ph^0 \cap \overline{B}_\varrho,\left(\ph^0 \cap \overline{B}_\varrho\right)\setminus\{0\}\right)=0 \quad \text{for all }k \geq 0.
    \end{align*}
    Hence,
    \begin{align*}
	C_k(\ph,0)=0 \quad \text{for all }k \geq 0.
    \end{align*}
    (see Section \ref{section_hypotheses}). This completes the proof.
\end{proof}

Thanks to Proposition \ref{proposition_Sec4_phi_zero} we can now establish the existence of extremal nontrivial constant sign solutions, that means, we will produce the smallest nontrivial positive solution and the greatest nontrivial negative solution of \eqref{problem}.

To this end, let $\mathcal{S}_+$ (resp. $\mathcal{S}_-$) be the set of all nontrivial positive (resp. negative) solutions of problem \eqref{problem}. As in Filippakis-Krist{\'a}ly-Papageorgiou \cite{2009-Filippakis-Kristaly-Papageorgiou}
we can show that
\begin{enumerate}
    \item[$\bullet$]
	$\mathcal{S}_+$ is downward directed, that means, if $u_1, u_2 \in \mathcal{S}_+$, then there exists $u\in \mathcal{S}_+$ such that $u \leq u_1$ and $u \leq u_2$.
    \item[$\bullet$]
	$\mathcal{S}_-$ is upward directed, that means, if $v_1, v_2 \in \mathcal{S}_-$, then there exists $v\in \mathcal{S}_-$ such that $v_1 \leq v$ and $v_2 \leq v$.
\end{enumerate}

By virtue of these lattice properties of $\mathcal{S}_+$ and $\mathcal{S}_-$ we see that for the purpose of producing extremal nontrivial constant sign solutions and since $\mathcal{S}_+ \subseteq \interior, \mathcal{S}_- \subseteq - \interior$, without any loss of generality, we may assume that there exists $M_{26}>0$ such that
\begin{align}\label{Sec4_67}
    \|u\|_{C(\close)} \leq M_{26} \quad \text{for all } u \in \mathcal{S}_+ \quad \text{and} \quad \|v\|_{C(\close)} \leq M_{26} \quad \text{for all } v \in \mathcal{S}_-.
\end{align}

Note that from hypotheses H(f)$_2$(i) and (iv) we find positive constants $a_1, a_2$ such that
\begin{align}\label{Sec4_31}
	f(x,s)s \geq a_1 |s|^\zeta-a_2 |s|^{r} \quad \text{for a.a. }x \in \Omega \text{ and for all }s\in \R.
\end{align}

This unilateral growth estimate leads to the following auxiliary Dirichlet problem
\begin{equation}\label{problem_aux}
    \begin{aligned}
      -\divergenz a(\nabla u(x)) & = a_1 |u|^{\zeta-2}u-a_2|u|^{r-2}u \quad && \text{in } \Omega,\\
       u & = 0  &&\text{on } \partial \Omega.
    \end{aligned}
\end{equation}
We are going to prove the uniqueness of constant sign solutions of \eqref{problem_aux}.

\begin{proposition}\label{proposition_Sec4_aux_problem_uniqueness}
    If hypotheses H(a)$_2$ hold, then problem \eqref{problem_aux} admits a unique nontrivial positive solution $u_* \in \interior$ and since \eqref{problem_aux} is odd, $v_*=-u_* \in \interior$ is the unique nontrivial negative solution of \eqref{problem_aux}.
\end{proposition}

\begin{proof}
    Let $\psi_+: \W1p0 \to \R$ be the $C^1$-functional defined by
    \begin{align*}
	\psi_+(u) = \int_\Omega G(\nabla u(x)) dx - \frac{a_1}{\zeta} \l\|u^+\r\|^\zeta_{\Lp{\zeta}}+\frac{a_2}{\hat{r}} \l\|u^+\r\|^{r}_{\Lp{r}}.
    \end{align*}
    Because of Corollary \ref{corollary_upper_lower_estimates} and due to $\zeta < p <r$ we observe that $\psi_+$ is coercive and in addition sequentially weakly lower semicontinuous. Then we find $u_* \in \W1p0$ such that
    \begin{align*}
	\psi_+(u_*)=\inf \left[\psi_+(u): u \in W^{1,p}_0(\Omega) \right]<0=\psi_+(0),
    \end{align*}
    since $\zeta<p<r$ (see the proof of Proposition \ref{proposition_Sec4_two_constant_sign}). Hence, $u_* \neq 0$. Moreover, as $u_*$ is the global minimizer of $\psi_+$ it holds $\left(\psi_+ \right)'(u_*)=0$ which means
    \begin{align}\label{Sec4_33}
	A(u_*)= a_1 \left(u^+_* \right)^{\zeta-1}-a_2 \left(u^+_* \right)^{r-1}.
    \end{align}
    Acting on (\ref{Sec4_33}) with $-u^-_* \in W^{1,p}_0(\Omega)$ and using Lemma \ref{lemma_properties}(c), we see that $u_* \geq 0$ and as before $u_* \neq 0$. Then, equation  \eqref{Sec4_33} becomes
    \begin{align*}
	A(u_*)=a_1u_*^{\zeta-1}-a_2 u_*^{r-1}
    \end{align*}
    and $u_*$ turns out to be a nontrivial positive solution of \eqref{problem_aux}. As before, the nonlinear regularity theory (see \cite{1968-Ladyzhenskaya-Ural'tseva}, \cite{1991-Lieberman}) implies that $u_* \in C^1_0(\close)$ and the nonlinear maximum principle of Pucci-Serrin \cite[pp. 111 and 120]{2007-Pucci-Serrin} yields that $u_* \in \interior$.

    We will complete the proof of the proposition if we prove the uniqueness of this solution $u_*$.
    To this end, let $\Psi_+ : L^1(\Omega) \to \R \cup \{\infty\}$ be the integral functional defined by
    \begin{align*}
	\Psi_+(u)=
	\begin{cases}
	    \displaystyle\int_\Omega G\left( \nabla u^{\frac{1}{q}}\right) dx \quad & \text{if } u \geq 0, u^{\frac{1}{q}} \in W^{1,p}_0(\Omega),\\
	    +\infty & \text{otherwise}.
	\end{cases}
    \end{align*}
    Take $u_1, u_2 \in \dom \Psi_+$ and let $u=(tu_1+(1-t)u_2)^{\frac{1}{q}}$ for $t \in [0,1]$. Applying Lemma 1 of D{\'{\i}}az-Sa{\'a} \cite{1987-Diaz-Saa} results in
    \begin{align*}
	\|\nabla u(x)\|_{\W1p0} \leq \left(t \left\|\nabla u_1(x)^{\frac{1}{q}}\right\|_{\W1p0}^q+(1-t) \left\|\nabla u_2(x)^{\frac{1}{q}}\right\|_{\W1p0}^{q} \right)^{\frac{1}{q}} \quad \text{a.e. in }\Omega.
    \end{align*}
    As $G_0$ is increasing and by means of H(a)$_2$(v) we conclude
    \begin{align*}
	& G_0\left(\left\|\nabla u(x)\right\|_{\W1p0} \right)\\
	& \leq G_0 \left(\left(t \left\|\nabla u_1(x)^{\frac{1}{q}}\right\|_{\W1p0}^q+(1-t) \left\|\nabla u_2(x)^{\frac{1}{q}}\right\|_{\W1p0}^{q} \right)^{\frac{1}{q}} \right)\\
	& \leq t G_0 \left( \left\|\nabla u_1(x)^{\frac{1}{q}}\right\|_{\W1p0}\right)+(1-t) G_0 \left( \left\|\nabla u_2(x)^{\frac{1}{q}}\right\|_{\W1p0} \right) \quad \text{a.e. in }\Omega.
    \end{align*}
    Note that by definition $G(\xi)=G_0(\|\xi\|)$ for all $\xi \in \R^N$. Hence
    \begin{align*}
	G(\nabla u(x)) \leq t G\left(\nabla u_1(x)^{\frac{1}{q}}\right) +(1-t) G\left(\nabla u_2(x)^{\frac{1}{q}}\right) \quad \text{a.e. in }\Omega,
    \end{align*}
    which proves that $\Psi_+$ is convex.

    Now we take two nontrivial positive solutions $v, w \in W^{1,p}_0(\Omega)$ of \eqref{problem_aux}. As mentioned before we know that $v,w$ belong to $\interior$. Therefore, $v, w \in \dom \Psi_+$. For $t \in (0,1)$ sufficiently small and $h \in C^1_0(\overline{\Omega})$ we have $v+th, w+th \in \dom \Psi_+$. Hence, $\Psi_+$ is Gateaux differentiable at $v$ and $w$ in the direction $h$. Furthermore, the chain rule yields
    \begin{align}
	& \Psi'_+\left(v^q\right)(h)=\frac{1}{q}\int_\Omega \frac{-\divergenz a(\nabla v)}{v^{q-1}}h dx \label{Sec4_34},\\
	& \Psi'_+\left(w^q\right)(h)=\frac{1}{q}\int_\Omega \frac{-\divergenz a(\nabla w)}{w^{q-1}}h dx. \label{Sec4_35}
    \end{align}
    Note that $\Psi_+'$ is monotone since $\Psi_+$ is convex. Then, from \eqref{Sec4_34} and \eqref{Sec4_35}, we derive
    \begin{align*}
	0
	& \leq \left \lan \Psi'_+\left(v^q \right)-\Psi'_+ \left(w^q\right),v^q-w^q \right \ran_{L^1(\Omega)}\\
	& = \frac{1}{q}\int_\Omega \left(\frac{-\divergenz a(\nabla v)}{v^{q-1}}+\frac{\divergenz a(\nabla w)}{w^{q-1}} \right) \left(v^q-w^q \right) dx\\
	& = \frac{1}{q}\int_\Omega \left(\frac{a_1 v^{\zeta-1}-a_2v^{r-1} }{v^{q-1}}- \frac{a_1w^{\zeta-1}-a_2 w^{r-1}}{w^{q-1}} \right)\left(v^q-w^q \right) dx\\
	& = \frac{a_1}{q} \into \left(\frac{1}{v^{q-\zeta}}-\frac{1}{w^{q-\zeta}} \right) \left(v^q-w^q \right)dx+\frac{a_2}{q} \into \left(w^{r-q}-v^{r-q} \right) \left(v^q-w^q \right)dx.
    \end{align*}
    Since $s \mapsto \frac{1}{s^{q-\zeta}}-s^{r-q}$ is strictly decreasing in $(0,\infty)$ we conclude that $v=w$ and thus, $u_* \in \interior$ is the unique nontrivial positive solution of \eqref{problem_aux}. Obviously, $v_*=-u_* \in - \interior$ is the unique nontrivial negative solution of \eqref{problem_aux}.
\end{proof}

\begin{proposition}\label{proposition_auxiliary}
    If hypotheses H(a)$_2$ and H(f)$_2$ hold, then $u_* \leq u$ for all $u \in \mathcal{S}_+$ and $v \leq v_*$ for all $v \in \mathcal{S}_-$ with $u_*, v_*$ being the nontrivial unique constant sign solutions of problem \eqref{problem_aux} obtained in Proposition \ref{proposition_Sec4_aux_problem_uniqueness}.
\end{proposition}

\begin{proof}
    Let $u \in \mathcal{S}_+$ and consider the Carath\'{e}odory function
    \begin{align}\label{Sec4_36}
	\vartheta_+(x,s)=
	\begin{cases}
	    0 \qquad & \text{if } s<0,\\
	    a_1s^{\zeta-1}-a_2 s^{r-1} & \text{if } 0 \leq s \leq u(x),\\
	    a_1 u(x)^{\zeta-1}-a_2u(x)^{r-1} & \text{if } u(x)<s.
	\end{cases}
    \end{align}
    We consider the $C^1$-functional $\Phi_+ : \W1p0 \to \R$ defined by
    \begin{align*}
	\Phi_+(u)=\int_\Omega G(\nabla u(x))dx-\int_\Omega \Theta_+(x,u(x))dx
    \end{align*}
    with $\Theta_+(x,s)=\int_0^s \vartheta_+(x,t)dt$. By means of the truncation it is clear that $\Phi_+$ is coercive and since it is also sequentially weakly lower semicontinuous there exists an element $\hat{u}_* \in W^{1,p}_0(\Omega)$ such that
    \begin{align*}
	\Phi_+(\hat{u}_*)=\inf \left[\Phi_+(u):u \in W^{1,p}_0(\Omega) \right ] <0=\Phi_+(0).
    \end{align*}
    as before since $\zeta<p<r$ (see the proof of Proposition \ref{proposition_Sec4_two_constant_sign}). Hence, $\hat{u}_* \neq 0$. Since $\hat{u}_*$ is a critical point of $\Phi_+$, we have
    \begin{align}\label{Sec4_37}
	A\left(\hat{u}_*\right)=N_{\vartheta_+}\l(\hat{u}_*\r).
    \end{align}
    Acting of \eqref{Sec4_37} with $-\hat{u}_*^- \in W^{1,p}_0(\Omega)$ we derive by applying Lemma \ref{lemma_properties}(c) that $\hat{u} \geq 0$. On the other side, acting with $\left(\hat{u}_*-u \right)^+ \in W^{1,p}_0(\Omega)$ on \eqref{Sec4_37}, there holds thanks to \eqref{Sec4_36}, \eqref{Sec4_31} and $u \in \mathcal{S}_+$,
    \begin{align*}
	\left \lan A \left(\hat{u}_* \right), \left(\hat{u}_*-u \right)^+ \right \ran
	& = \int_\Omega \vartheta_+\left(x,\hat{u}_* \right)\left(\hat{u}_*-u \right)^+  dx\\
	& = \int_\Omega \left(a_1 u^{\zeta-1}-a_2 u^{r-1} \right) \left(\hat{u}_*-u \right)^+dx\\
	& \leq \int_\Omega f(x,u)\left(\hat{u}_*-u \right)^+dx\\
	& = \left \lan A \left( u \right), \left(\hat{u}_*-u \right)^+ \right \ran.
    \end{align*}
    This gives
    \begin{align*}
	\int_{\{\hat{u}_* > u\}} \left(a\left(\nabla \hat{u}_*\right)-a\left(\nabla u\right),\nabla \hat{u}_*-\nabla u \right)_{\R^N}dx \leq 0.
    \end{align*}
    Since $a$ is strictly monotone (see Lemma \ref{lemma_properties}(a)) we obtain $\l|\{\hat{u}_* > u\} \r|_N=0$. To sum up, we have
    \begin{align*}
	0 \neq \hat{u}_* \in \left[0, u \right]= \left \{v \in W^{1,p}_0(\Omega): 0 \leq v(x) \leq u(x) \text{ a.e. in } \Omega \right\}.
    \end{align*}
    By definition of the truncation in \eqref{Sec4_36} it follows $\vartheta_+(x,\hat{u}_*)=a_1\hat{u}_*^{\zeta-1}-a_2 \hat{u}_*^{r-1}$. Therefore, $\hat{u}_*$ solves the auxiliary problem \eqref{problem_aux} but Proposition \ref{proposition_Sec4_aux_problem_uniqueness} proved the uniqueness of constant sign solutions of \eqref{problem_aux}. We deduce that $\hat{u}_*=u_*\in \interior$ and $u_* \leq u$. Since $u \in \mathcal{S}_+$ was arbitrary we deduce that
    \begin{align*}
	u_* \leq u \quad \text{for all } u \in \mathcal{S}_+.
    \end{align*}
    Similarly, we prove that $v \leq v_*$ for all $v \in \mathcal{S}_-$.
\end{proof}

Now we are ready to produce extremal nontrivial constant sign solutions of our original problem \eqref{problem}.

\begin{proposition}\label{proposition_Sec4_extremal_solutions}
    Under the assumption H(a)$_2$ and H(f)$_2$ problem \eqref{problem} possesses a smallest positive solution $u_+ \in \interior$ and a greatest negative solution $v_- \in - \interior$.
\end{proposition}

\begin{proof}
    Let $\mathcal{C} \subseteq \mathcal{S}_+$ be a chain, i.e., a totally ordered subset of $\mathcal{S}_+$. Then there is a sequence  $(u_n)_{n \geq 1} \subseteq \mathcal{S}_+$ such that
    \begin{align*}
	\inf \mathcal{C}= \inf_{n \geq 1} u_n.
    \end{align*}
    (see Dunford-Schwartz \cite[p. 336]{1958-Dunford-Schwartz}). Since $u_n \in \mathcal{S}_+$ we have
    \begin{align}\label{Sec4_39}
	A(u_n)= N_f(u_n) \quad\text{ for all }n \geq 1.
    \end{align}
    Therefore, thanks to \eqref{Sec4_67}, H(f)$_2$(i) and Lemma \ref{lemma_properties}, we observe that $(u_n)_{n\geq 1}\subseteq W^{1,p}_0(\Omega)$ is bounded and we may assume that
    \begin{align}\label{Sec4_40}
	u_n \weak u \text{ in } W^{1,p}_0(\Omega) \quad \text{ and } \quad u_n \to u \text{ in } L^p(\Omega).
    \end{align}
    Acting on \eqref{Sec4_39} with $u_n-u \in W^{1,p}_0(\Omega)$ and making use of \eqref{Sec4_40} yields
    \begin{align*}
	\lim_{n \to \infty} \left \lan A(u_n),u_n-u \right \ran=0.
    \end{align*}
    Therefore, the (S$_+$)-property of $A$ (see Proposition \ref{proposition_basic_properties}) gives $u_n\to u$ in $W^{1,p}_0(\Omega)$. Passing to the limit in \eqref{Sec4_39} we get
     \begin{align}\label{Sec4_68}
	A(u)= N_f(u).
    \end{align}
    Taking into account Proposition \ref{proposition_auxiliary} provides $u_* \leq u_n$ for all $n \geq 1$ which implies $u_* \leq u$ and with regard to \eqref{Sec4_68} $u \in \mathcal{S}_+$. Furthermore, we have $u=\inf \mathcal{C}$. Since $\mathcal{C}$ was arbitrarily chosen in $\mathcal{S}_+$ the Kuratowski-Zorn Lemma ensures that $\mathcal{S}_+$ has a minimal element $u_+ \in \mathcal{S}_+$. Since $\mathcal{S}_+$ is downward directed we conclude that $u_+ \in \interior$ is the smallest nontrivial positive solution of (\ref{problem}).

    Working with $\mathcal{S}_-$ instead of $\mathcal{S}_+$ and applying again the Kuratowski-Zorn Lemma, we can show that $v_- \in -\interior$ is the greatest nontrivial negative solution of (\ref{problem}). Recall that $\mathcal{S}_-$ is upward directed.
\end{proof}

Having these extremal nontrivial constant sign solutions, we are now in the position to produce a nodal (sign changing) solution of problem \eqref{problem}.

\begin{proposition}\label{proposition_Sec4_nodal_solution}
    Let H(a)$_2$ and H(f)$_2$ be satisfied. Then problem \eqref{problem} has a nodal solution $y_0 \in [v_-,u_+]\cap C^1_0(\close)$.
\end{proposition}

\begin{proof}
    By reason of Proposition \ref{proposition_Sec4_extremal_solutions} we know that $u_+ \in \interior$ and $v_- \in - \interior$ are the extremal nontrivial constant sign solutions of \eqref{problem}. With the aid of these extremal solutions we introduce the cut-off function $f_0 : \Omega \times \R \to \R$
    \begin{align}\label{Sec4_41}
	f_0(x,s)=
	\begin{cases}
	    f\l(x,v_-(x) \r) \quad & \text{if } s<v_-(x)\\
	    f\l(x,s\r) & \text{if } v_-(x) \leq s \leq u_+(x)\\
	    f\l(x,u_+(x)\r) & \text{if } u_+(x)<s
	\end{cases},
    \end{align}
    which is clearly a Carath\'{e}odory function. For $F_0(x,s)=\int_0^s f_0(x,t)dt$ we define the $C^1$-functional $\varphi_0 : \W1p0 \to \R$ by
    \begin{align*}
	\varphi_0 (u)=\int_\Omega G\l(\nabla u(x)\r) dx-\int_\Omega F_0(x,u(x)) dx.
    \end{align*}
    For $f_0^{\pm}(x,s)=f_0 \l(x,\pm s^{\pm}\r)$ we also consider the functionals $\varphi^{\pm}_0: \W1p0 \to \R$
    \begin{align*}
	\varphi_0^{\pm}(u)=\int_\Omega G\l(\nabla u(x)\r) dx-\int_\Omega F_0^{\pm}(x,u(x)) dx
    \end{align*}
    with $F_0^{\pm}(x,s)=\int_0^s f_0^{\pm}(x,t)dt$. 
    
    As in the proof of Proposition \ref{proposition_auxiliary} it can be easily shown that
    \begin{align*}
	K_{\varphi_0} \subseteq [v_-,u_+], \quad K_{\varphi_0^+}\subseteq \l[0,u_+\r], \quad  K_{\varphi_0^-}\subseteq \l[v_-,0\r].
    \end{align*}
    Then, the extremality properties of $u_+ \in \interior$ and $v_- \in - \interior$ imply that
    \begin{align}\label{claim_1}
	K_{\varphi_0} \subseteq [v_-,u_+], \quad K_{\varphi_0^+}=\l\{0,u_+\r\}, \quad K_{\varphi_0^-}=\l\{v_-,0\r\}.
    \end{align}

    {\bf Claim:} $u_+ \in \interior$ and $v_- \in -\interior$ are local minimizers of $\varphi_0$.

    First note that $\varphi_0^+$ is coercive (see \eqref{Sec4_41}) and sequentially weakly lower semicontinuous. Then there exists $\hat{u} \in W^{1,p}_0(\Omega)$ such that
    \begin{align*}
        \varphi_0^+\l (\hat{u} \r)=\inf \l \{ \varphi_0^+(u): u \in W^{1,p}_0(\Omega)\r\}.
    \end{align*}
    Similar to the proof of Proposition \ref{proposition_Sec4_two_constant_sign} (see \eqref{Sec4_2}) we have $\varphi_0^+\l(\hat{u}\r)<0=\varphi_0^+(0)$, hence $\hat{u} \neq 0$.
    Then, \eqref{claim_1} implies $\hat{u}=u_+ \in \interior$. Since $ \varphi_0\big|_{C^1_0(\close)_+}=\varphi_0^+\big|_{C^1_0(\close)_+}$ we deduce that $u_+ \in \interior$ is a local $C^1_0(\overline{\Omega})$-minimizer of $\varphi_0$ and thanks to Proposition \ref{proposition_local_minimizers} it follows that $u_+$ is a local $W^{1,p}_0(\Omega)$-minimizer of $\varphi_0$. The assertion for $v_- \in - \interior$ can be shown similarly, using $\varphi_0^-$ instead of $\varphi_0^+$. This proves the claim.

    We may assume, without loss of generality, that $\varphi_0(v_-) \leq \varphi_0(u_+)$. By virtue of the claim, we find a number $\rho \in (0,1)$ such that
    $\|v_--u_+\|_{W^{1,p}_0(\Omega)}>\rho$ and
    \begin{align}\label{Sec4_43}
        \varphi_0\l(v_-\r) \leq \varphi_0\l(u_+\r)<\inf \l [ \varphi_0(u):\|u-u_+\|_{W^{1,p}_0(\Omega)}= \rho \r]=\eta_0.
    \end{align}
    (see Aizicovici-Papageorgiou-Staicu \cite[proof of Proposition 29]{2008-Aizicovici-Papageorgiou-Staicu}). Because of the definition of the truncation in \eqref{Sec4_41} it is clear that $\varphi_0$ is coercive and so it satisfies the C-condition. This fact in conjunction with \eqref{Sec4_43} permits the usage of the mountain pass theorem stated in Theorem \ref{theorem_mountain_pass}. Therefore, we find $y_0 \in W^{1,p}_0(\Omega)$ such that
    \begin{align}\label{Sec4_44}
	y_0 \in K_{\varphi_0} \subseteq [v_-,u_+] \quad \text{and} \quad \eta_0 \leq \ph_0(y_0)
    \end{align}
    (see also \eqref{claim_1}). From \eqref{Sec4_44}, \eqref{Sec4_41}, and \eqref{Sec4_43} it follows that $y_0$ is a solution of \eqref{problem}  and $y_0 \not\in \{v_-,u_+\}$. The nonlinear regularity theory implies that $y_0 \in C^1_0(\close)$.
    
    Since $y_0$ is a critical point of $\ph_0$ of mountain pass type, we have
    \begin{align}\label{Sec4_69}
        C_1(\varphi_0,y_0) \neq 0.
    \end{align}
    On the other side Proposition \ref{proposition_Sec4_phi_zero} amounts
    \begin{align*}
        C_k(\varphi,0)=0 \quad \text{for all } k \geq 0.
    \end{align*}
    Moreover, \eqref{Sec4_41} implies $\ph\big|_{[v_-,u_+]}=\ph_0\big |_{[v_-,u_+]}$ and since $u_+ \in \interior, v_- \in - \interior$ combined with the homotopy invariance of critical groups (cf. the proof of Proposition \ref{proposition_phi_u0_v0}) we infer that
    \begin{align}\label{Sec4_70}
	C_k(\ph_0,0)=C_k(\ph,0)=0 \quad \text{for all } k \geq 0.
    \end{align}
    Comparing \eqref{Sec4_69} and \eqref{Sec4_70} we obtain that $y_0 \in [v_-,u_+] \cap C^1_0(\close) \setminus \{0\}$. Due to the extremality of $u_+$ and $v_-$ the solution $y_0$ must be nodal.
\end{proof}

Summarizing this section we can state the following multiplicity theorem for problem \eqref{problem}.

\begin{theorem}\label{theorem_Sec4_main_result}
    If hypotheses H(a)$_2$ and H(f)$_2$ hold, then problem (\ref{problem}) has at least four constant sign solutions
    \begin{align*}
	& \bullet u_0, u_1 \in \interior, u_0 \leq u_1, u_0 \neq u_1\\
	& \bullet v_0, v_1 \in -\interior, v_1 \leq v_0, v_1 \neq v_0
    \end{align*}
    and at least one sign-changing (nodal) solution
    \begin{align*}
	y_0 \in [v_0,u_0] \cap C^1_0(\close).
    \end{align*}
\end{theorem}

\begin{proof}
    The result follows from the Propositions \ref{proposition_Sec4_two_constant_sign}, \ref{proposition_Sec4_four_constant_sign}, and \ref{proposition_Sec4_nodal_solution}.
\end{proof}

In the next section we will improve Theorem \ref{theorem_Sec4_main_result} for a particular case of problem \eqref{problem} and with stronger regularity conditions on the reaction $f(x,\cdot)$. It will be shown the existence of a second nodal solution for a total of six nontrivial solutions given with complete sign information.

\section{$(p,2)$-equation}\label{section_five}

In this section we deal with a particular case of problem \eqref{1}. Namely, we assume that
\begin{align*}
    a(\xi)=\|\xi\|^{p-2}\xi+\xi \quad \text{for all }\xi \in \R^N \text{ with } 2 \leq p <\infty.
\end{align*}
In this case the differential operator becomes the $(p,2)$-Laplacian, that is
\begin{align*}
    \divergenz a(\nabla u)=\Delta_p u + \Delta u \quad \text{for all } u \in \W1p0.
\end{align*}

This differential operator arises in problems of quantum physics in connection with Derick's model \cite{1964-Derrick} for the existence of solitons (see Benci-D'Avenia-Fortunato-Pisani \cite{2000-Benci-D'Avenia-Fortunato-Pisani}).

Therfore, the problem under consideration is the following:
\begin{equation}\label{problem2}
    \begin{aligned}
      -\Delta_p u(x)-\Delta u(x) & = f(x,u(x)) \quad && \text{in } \Omega,\\
       u & = 0  &&\text{on } \partial \Omega.
    \end{aligned}
\end{equation}

Under stronger regularity conditions on the reaction $f(x,\cdot)$ we will show that problem \eqref{problem2} has a second nodal solution for a total of six nontrivial solutions (two positive, two negative, and two nodal).

We need to strengthen our hypotheses on the perturbation $f:\Omega \times \R \to \R$ in the following way.

\begin{enumerate}
    \item[H(f)$_3$]
	$f: \Omega \times \R \to \R$ is a measurable function such that $f(x,0)=0, f(x, \cdot) \in C^1(\R)$ for a.a. $x \in \Omega$, hypotheses H(f)$_3$(ii), (iii), (v), (vi) are the same as the corresponding hypotheses H(f)$_2$(ii), (iii), (v), (vi) and
	\begin{enumerate}
	    \item[(i)]
		$|f'_s(x,s)| \leq a(x) \left(1+|s|^{r-2}\right)$ for a.a. $x \in \Omega$, for all $s \in \R$, with $a\in \Linf_+$, and $2<r<p^*$;
	    \item[(iv)]
		$f'_s(x,0)=\lim_{s \to 0} \frac{f(x,s)}{s}$ uniformly for a.a. $x\in\Omega$, $$f'_s(x,0) \in \l[\hat{\lambda}_m(2),\hat{\lambda}_{m+1}(2)\r] \quad \text{a.e. in } \Omega \text{ with }m \geq 2,$$ and $f_s'(\cdot,0)\neq \hat{\lambda}_m(2)$, $f'_s(\cdot,0) \neq \hat{\lambda}_{m+1}(2)$.
	\end{enumerate}
\end{enumerate}

\begin{remark}
	Note that the asymptotic behavior of $f(x,\cdot)$ at $\pm\infty$ remains the same. The situation has changed near zero (see H(f)$_3$(iv)) since the concave term has power equal to $q=2$ (i.e. $\zeta=q=2$). This changes the computation of the critical groups of the energy functional $\ph$ at the origin.
\end{remark}

\begin{example}
	The following function satisfies hypotheses H(f)$_3$ (the $x$-dependence is dropped again):
	\begin{align*}
		f(x)=
		\begin{cases}
			\lambda s- c s^2 &\text{if }|s|\leq 1\\
			\beta s \ln |s|-(\lambda-c)|s|^{\frac{1}{2}}& \text{if }|s|> 1
		\end{cases}
	\end{align*}
	with $\lambda \in \left( \hat{\lambda}_m(2),\hat{\lambda}_{m+1}(2)\right)$ for some $m\geq 2$, $\beta>4\lambda$, and $c=\frac{2\beta-\lambda}{5}>0$.
\end{example}

We start with the computation of the critical groups at the origin.

\begin{proposition}\label{proposition_Sec5_phi_zero}
	Let hypotheses H(f)$_3$ be satisfied. Then
	\begin{align*}
		C_k(\ph,0)=\delta_{k,d_m} \Z \quad \text{for all }k \geq 0
	\end{align*}
	with $d_m=\dim \oplus_{i=1}^m E\left( \hat{\lambda}_i(2)\right)\geq 2$.
\end{proposition}

\begin{proof}
	Consider the $C^2$-functional $\gamma: \W1p0 \to \R$ defined by
	\begin{align*}
		\gamma(u)=\frac{1}{p} \|\nabla u\|_{\Lp{p}}^p
		+\frac{1}{2} \|\nabla u\|_{\Lp{2}}^2-\frac{1}{2} \into f'_u(x,0)u^2 dx.
	\end{align*}
	By virtue of hypothesis H(f)$_3$(iv), given $\eps>0$, there exists $\delta=\delta(\eps) \in (0,1)$ such that
	\begin{align*}
		\left|\frac{f(x,s)}{s}-f'_s(x,0) \right| \leq \eps \quad \text{for a.a. }x\in \Omega \text{ and for all } 0<|s|\leq \delta,
	\end{align*}
	which implies that
	\begin{align*}
		\left|F(x,s)-\frac{1}{2}f'_s(x,0)s^2 \right| \leq \eps \quad \text{for a.a. }x\in \Omega \text{ and for all } 0<|s|\leq \delta.
	\end{align*}
	Therefore, we find $\varrho \in (0,1)$ such that
	\begin{align*}
		\|\ph-\gamma\|_{C^1_0 \left(\overline{B}_\varrho^C \right)} \leq \eps,
	\end{align*}
	where $\overline{B}_\varrho^C= \left \{ u \in C^1_0(\close): \|u\|_{C^1_0(\close)} \leq \varrho \right\}$.
	
	Choosing $\eps>0$ sufficiently small gives
	\begin{align*}
		C_k\left( \ph \big |_{C^1_0(\close)},0\right)=C_k\left( \gamma \big |_{C^1_0(\close)},0\right) \quad \text{for all }k\geq 0
	\end{align*}
	(see Chang \cite[p. 336]{2005-Chang}) and since $C^1_0(\close)$ is dense in $\W1p0$ it follows
	\begin{align}\label{Sec5_1}
		C_k(\ph,0)=C_k(\gamma,0) \quad \text{for all }k\geq 0
	\end{align}
	(see Palais \cite{1966-Palais}). Moreover, due to Cingolani-Vannella \cite[Theorem 1]{2003-Cingolani-Vannella}, one has
	\begin{align*}
		C_k(\gamma,0)=\delta_{k,d_m}\Z \quad \text{for all }k\geq 0,
	\end{align*}
	which, because of \eqref{Sec5_1}, results in
	\begin{align*}
		C_k(\ph,0)=\delta_{k,d_m}\Z \quad \text{for all }k\geq 0.
	\end{align*}
\end{proof}

A careful inspection of the proofs in the previous section reveals that the results remain valid although we have a different geometry near zero (since $\zeta=q=2$ in the notation of Section \ref{section_four}). In this case, by means of hypotheses H(f)$_3$(i), (iv), we know that for given $\eps>0$ there is a number $M_{27}=M_{27}(\eps)>0$ such that
\begin{align*}
	f(x,s)s \geq \left(f'_s(x,0)-\eps\right)s^2-M_{27}|s|^r \quad \text{for a.a. }x\in\Omega \text{ and for all }s\in\R.
\end{align*}
This unilateral growth estimate leads to the following auxiliary Dirichlet problem
\begin{equation}\label{problem_aux2}
    \begin{aligned}
      -\Delta_p u(x)-\Delta u(x) & = \left(f'_u(x,0)-\eps\right)u(x)^2-M_{27}|u(x)|^{r-2}u(x) \quad && \text{in } \Omega,\\
       u & = 0  &&\text{on } \partial \Omega.
    \end{aligned}
\end{equation}

Choosing $\eps \in\left(0,\hat{\lambda}_m(2)-\hat{\lambda}_{m+1}(2)\right)$ we can show that problem \eqref{problem_aux2} admits a unique nontrivial positive solution $u_*\in \interior$ and, by the oddness of \eqref{problem_aux2}, we have that $v_*=-u_*\in -\interior$ is the unique nontrivial negative solution of \eqref{problem_aux2}. The proof can be done as the proof of Proposition \ref{proposition_Sec4_aux_problem_uniqueness}. Therefore, the arguments of Section \ref{section_four} apply and we produce five nontrivial solutions
\begin{align*}
    \begin{split}
	& \bullet u_0, u_1 \in \interior, u_0 \leq u_1, u_0 \neq u_1;\\
	& \bullet v_0, v_1 \in -\interior, v_1 \leq v_0, v_1 \neq v_0;\\
  	& \bullet y_0 \in [v_0,u_0] \cap C^1_0(\close) \text{ nodal}.
    \end{split}
\end{align*}

Using these five solutions and Morse theory, we can produce a sixth nontrivial solution being nodal.

\begin{theorem}\label{theorem_Sec5_main_result}
	Let hypotheses H(f)$_3$ be satisfied. Then problem (\ref{problem2}) has at least six nontrivial solutions
	 \begin{align*}
		& \bullet u_0, u_1 \in \interior, u_0 \leq u_1, u_0 \neq u_1;\\
		& \bullet v_0, v_1 \in -\interior, v_1 \leq v_0, v_1 \neq v_0;\\
  		& \bullet y_0, y_1 \in \ints_{C^1_0(\close)} [v_0,u_0]  \text{ nodal}.
 	\end{align*}
\end{theorem}

\begin{proof}
    As we already remarked the conclusion of Theorem \ref{theorem_Sec4_main_result} remains valid in the present setting and thus we already have five nontrivial solutions
    \begin{align*}
	\begin{split}
	    & \bullet u_0, u_1 \in \interior, u_0 \leq u_1, u_0 \neq u_1;\\
	    & \bullet v_0, v_1 \in -\interior, v_1 \leq v_0, v_1 \neq v_0;\\
	    & \bullet y_0 \in [v_0,u_0] \cap C^1_0(\close) \text{ nodal}.
	\end{split}
    \end{align*}
    	
    Without loss of generality we may assume that both, $u_0$ and $v_0$, are extremal nontrivial constant sign solutions, i.e., $u_0=u_+$ and $v_0=v_-$ in the notation of Proposition \ref{proposition_Sec4_extremal_solutions}. We have
    \begin{align*}
	-\Delta_p u_0(x)-\Delta u_0(x)-f(x,u_0(x))=0=-\Delta_p y_0(x)-\Delta y_0(x)-f(x,y_0(x)
    \end{align*}
    for a.a. $x \in \Omega$ and $y_0 \leq u_0$. As $a(\xi)=\|\xi\|^{p-2}\xi+\xi$ for all $\xi \in \R^N$ we see that $a \in C^1(\R^N,\R^N)$. Hence,
    \begin{align*}
	\nabla a(\xi)=\|\xi\|^{p-2} \l(I+ (p-2)\frac{ \xi\otimes\xi}{\|\xi\|^2} \r)+I \quad \text{ for all } \xi \in \R^N \setminus \{0\},
    \end{align*}
    implying
    \begin{align*}
	\l(\nabla a(\xi)y,y \r)_{\R^N} \geq \|y\|^2 \quad \text{for all }\xi, y  \in \R^N.
    \end{align*}
    This fact along with hypothesis H(f)$_3$(iv) permits the usage of the tangency principle of Pucci-Serrin \cite[p. 35]{2007-Pucci-Serrin} to obtain $y_0(x)<u_0(x)$ for all $x \in \Omega$. Similarly, one can prove $v_0(x)<y_0(x)$ for all $x \in \Omega$.
    
    Let $\varrho=\max \left \{\|u_0\|_{C(\close)}, \|v_0\|_{C(\close)} \right\}$ and let $\xi_\varrho$ be as postulated in hypothesis H(f)$_3$(vi). For $\xi>\xi_\varrho$ we infer
    \begin{align*}
	& -\Delta_pu_0(x)-\Delta u_0(x)+\xi u_0(x)^{p-1}\\
	& =f(x,u_0(x))+\xi u_0(x)^{p-1}\\
	& = f(x,u_0(x))+\xi_\varrho u_0(x)^{p-1}+ \left(\xi-\xi_\varrho \right) u_0(x)^{p-1}\\
	& \geq f(x,y_0(x))+\xi_\varrho |y_0(x)|^{p-2}y_0(x)+ \left(\xi-\xi_\varrho \right) u_0(x)^{p-1}\\
	& > f(x,y_0(x))+\xi_\varrho |y_0(x)|^{p-2}y_0(x)+ \left(\xi-\xi_\varrho \right) |y_0(x)|^{p-2}y_0(x)\\
	& = -\Delta_p y_0(x)-\Delta y_0(x)+\xi|y_0(x)|^{p-2}y_0(x) \quad \text{a.e. in }\Omega.
    \end{align*}
    Since $u_0 \in \interior$ and $y_0 \in C^1_0(\close)$ we may apply the strong comparison principle of Papageorgiou-Smyrlis \cite[Proposition 3]{2013-Papageorgiou-Smyrlis} and deduce that $u_0-y_0 \in \interior$. In a similar fashion we show that $y_0-v_0 \in \interior$. Therefore, we have proved that
    \begin{align}\label{Sec5_3}
	y_0 \in \ints_{C^1_0(\overline{\Omega})} [v_0,u_0].
    \end{align}
    
    Let $\varphi_0 \in C^{2-0}\left( \W1p0\right)$ be the functional introduced in the proof of Proposition \ref{proposition_Sec4_nodal_solution} by truncating the reaction $f(x,\cdot)$ at $\{v_0(x),u_0(x)\}$. Recall that
    \begin{align}\label{Sec5_4}
	C_1(\ph_0,y_0) \neq 0
    \end{align}
    (see \eqref{Sec4_69}). The homotopy invariance of critical groups along with \eqref{Sec5_3} gives
    \begin{align}\label{Sec5_7}
	C_k(\ph_0,y_0)=C_k(\ph,y_0) \quad \text{for all }k \geq 0,
    \end{align}
    (see the proof of Proposition \ref{proposition_phi_u0_v0}) which implies, due to \eqref{Sec5_4},
    \begin{align*}
	C_1(\ph,y_0) \neq 0.
    \end{align*}
    Since $\ph \in C^2 \left( \W1p0 \right)$, from Papageorgiou-Smyrlis \cite[the proof of Proposition 12, Claim 2]{2013-Papageorgiou-Smyrlis}, we infer that
    \begin{align*}
	C_k(\ph,y_0)= \delta_{k,1}\Z \quad \text{for all }k \geq 0,
    \end{align*}
    which implies, because of \eqref{Sec5_7},
    \begin{align}\label{Sec5_10}
	C_k(\ph_0,y_0)= \delta_{k,1}\Z \quad \text{for all }k \geq 0.
    \end{align}
    Recall that $u_0 \in \interior$ and $v_0 \in -\interior$ are local minimizers of $\ph_0$ (see the claim in the proof of Proposition \ref{proposition_Sec4_nodal_solution}). Hence, we get
    \begin{align}\label{Sec5_11}
	C_k(\ph_0,u_0)=C_k (\ph_0,v_0)=\delta_{k,0} \Z \quad \text{for all }k \geq 0.
    \end{align}
    Since $\ph_0 \big|_{[v_0,u_0]}=\ph\big|_{[v_0,u_0]}$, $u_0 \in \interior, v_0 \in -\interior$, Proposition \ref{proposition_Sec5_phi_zero}, and the homotopy invariance of critical groups we see that
    \begin{align}\label{Sec5_12}
	C_k(\ph_0,0)=\delta_{k,d_m} \Z \quad \text{for all } k \geq 0.
    \end{align}
    Finally, by means of the truncation defined in \eqref{Sec4_41}, it is easy to see that $\ph_0$ is coercive. Therefore
    \begin{align}\label{Sec5_13}
	C_k\left(\ph_0,\infty \right)=\delta_{k,0}\Z \quad \text{for all }k \geq 0.
    \end{align}
    Now suppose that $K_{\ph_0}=\{0,u_0,v_0,y_0\}$. Taking into account the Morse relation given in \eqref{morse_relation} by setting $t=-1$ combined with \eqref{Sec5_10}--\eqref{Sec5_13} results in
    \begin{align*}
	(-1)^{d_m}+2(-1)^0+(-1)^1=(-1)^0,
    \end{align*}
    which gives the contradiction $(-1)^{d_m}=0$. Hence, we can find another $y_1 \in K_{\ph_0}$ satisfying $y_1 \not\in \{0,u_0,v_0,y_0\}$. Due to \eqref{claim_1} we know that $K_{\ph_0} \subseteq [u_0,v_0]$ and as we supposed that $u_0,v_0$ are the extremal solutions of \eqref{problem2}, it follows that $y_1$ is a nodal solution of \eqref{problem2} distinct from $y_0$. Finally, the
    usage of the nonlinear regularity theory implies that $y_1 \in C^1_0(\close)$. Moreover, similar to $y_0$ (see \eqref{Sec5_3}), we can show that
    \begin{align*}
	y_1 \in \ints_{C^1_0(\close)}[v_0,u_0].
    \end{align*}
    The proof is complete.
\end{proof}

\section{Nonlinear eigenvalue problem}\label{section_six_solutions}

In this section we deal with the following nonlinear eigenvalue problem
\begin{equation}\tag{P$_\lambda$}\label{problem3}
    \begin{aligned}
      -\divergenz (\nabla u(x))& = \lambda f(x,u(x)) \quad && \text{in } \Omega,\\
       u & = 0  &&\text{on } \partial \Omega.
    \end{aligned}
\end{equation}

As before, the reaction $f: \Omega \times \R \to \R$ is supposed to be a Carath\'{e}odory function which exhibits $(p-1)$-superlinear growth near $\pm\infty$ without satisfying the Ambrosetti-Rabinowitz condition. Our aim is to prove that problem \eqref{problem3} admits at least two nontrivial solutions provided $\lambda>0$ is sufficiently small. Moreover, one of these solutions vanishes as $\lambda \to 0^+$ and the other one blows up as $\lambda \to 0^+$, both in the Sobolev norm $\|\cdot\|_{\W1p0}$.

We suppose the following conditions on the reaction $f: \Omega \times \R \to \R$.
\begin{enumerate}[leftmargin=1.2cm]
    \item[H(f)$_4$:]
	$f: \Omega \times \R \to \R$ is a Carath\'{e}odory function satisfying $f(x,0)=0$, $f(x,s) \geq 0$ for a.a. $x \in \Omega$ and for all $s \geq 0$ such that
	\begin{enumerate}
	    \item[(i)]
		$|f(x,s)| \leq a(x) \left(1+|s|^{r-1}\right)$ for a.a. $x \in \Omega$, for all $s\geq 0$, with $a \in \Linf_+$, and $p < r<p^*$;
	    \item[(ii)]
		if $F(x,s)=\int_0^sf(x,t)dt$, then
		\begin{align*}
		    \lim_{s \to +\infty} \frac{F(x,s)}{s^{p}}=+\infty \quad  \text{uniformly for a.a. } x \in \Omega;
		\end{align*}
	    \item[(iii)]
		there exist $\tau \in \left((r-p)\max \left \{ \frac{N}{p},1\right \},p^*\right)$ and $\beta_0>0$ such that
		\begin{align*}
		    \liminf_{s \to +\infty} \frac{f(x,s)s-pF(x,s)}{s^{\tau}} \geq \beta_0 \quad \text{uniformly for a.a. } x \in \Omega;
		\end{align*}
	    \item[(iv)]
		there exist $\zeta \in (1,q)$ ($q$ as in hypothesis H(a)$_2$(v)) and $\delta>0$ such that
		\begin{align*}
		    \zeta F(x,s) \geq f(x,s)s > 0 \quad \text{for a.a. }x\in \Omega, \text{ for all }0<s \leq \delta,
		\end{align*}
		and
		\begin{align*}
		    \essinf_\Omega F(\cdot,\delta)>0;
		\end{align*}
	    \item[(v)]
		for every $\varrho>0$ there exists $\xi_\varrho>0$ such that
		\begin{align*}
		    s\mapsto f(x,s)+\xi_\varrho s^{p-1}
		\end{align*}
		is nondecreasing on $[0,\varrho]$ for a.a. $x\in\Omega$.
	\end{enumerate}
\end{enumerate}

\begin{remark}
    Since we are looking for positive solutions and as the hypotheses above concern the positive semiaxis $\R_+=[0,\infty)$, without loss of generality, we may assume that $f(x,s)=0$ for a.a. $x \in \Omega$ and for all $s\leq 0$.
\end{remark}

We have the following existence theorem for problem \eqref{problem3}.

\begin{theorem}
    Assume H(a)$_2$ and H(f)$_4$. Then there exists $\lambda^*>0$ such that problem \eqref{problem3} possesses at least two solutions $u_\lambda, v_\lambda \in \interior$ for all $\lambda \in (0,\lambda^*)$ satisfying
    \begin{align*}
	\l\|u_\lambda\r\|_{\W1p0} \to \infty \quad \text{ and } \quad \l\|v_\lambda\r\|_{\W1p0} \to 0 \quad \text{ as } \lambda \to 0^+.
    \end{align*}
\end{theorem}

\begin{proof}
    Let $\varphi_\lambda : \W1p0 \to \R$ be the $C^1$-energy functional of problem \eqref{problem3} defined by
    \begin{align*}
	\varphi_\lambda(u)=\into G(\nabla u(x)) dx - \lambda \into F(x,u(x)) dx
    \end{align*}
    with $F(x,s)=\int^s_0 f(x,t)dt$. By means of H(f)$_4$(i) and (iv) we obtain the estimate
    \begin{align}\label{sec6_1}
	|F(x,s)| \leq \frac{M_{28}}{q}\l(s^+\r)^{q}+\frac{M_{29}}{r} \l(s^+\r)^r \quad \text{for a.a. }x \in \Omega \text{ and for all } s\in \R
    \end{align}
    with positive constants $M_{28}$ and $M_{29}$. Taking into account Corollary \ref{corollary_upper_lower_estimates}, \eqref{sec6_1}, and the Sobolev embedding theorem gives
    \begin{align}\label{sec6_2}
	\begin{split}
	    \varphi_\lambda(u)=
	    & \into G(\nabla u) dx - \lambda \into F(x,u) dx\\
	    & \geq \frac{c_1}{p(p-1)} \|u\|_{\W1p0}^p- \lambda \left[\frac{M_{28}}{q}\|u\|^{q}_{\Lp{q}}+\frac{M_{29}}{r}\|u\|^r_{\Lp{r}} \right]\\
	    & \geq \frac{c_1}{p(p-1)} \|u\|_{\W1p0}^p- \lambda \left[M_{30}\|u\|^{q}_{\W1p0}+M_{31} \|u\|^r_{\W1p0} \right]
	\end{split}
    \end{align}
    for all $u \in \W1p0$ and with positive constants $M_{30}, M_{31}$ both independent of $\lambda>0$. Now let $\alpha \in \left(0, \frac{1}{r-p} \right)$ and suppose that $\|u\|_{\W1p0}=\lambda^{-\alpha}$. Then, \eqref{sec6_2} reads as
    \begin{align}\label{sec6_3}
	\begin{split}
	    \varphi_\lambda(u)
	    & \geq \frac{c_1}{p(p-1)} \lambda^{-\alpha p}- M_{30} \lambda^{1-\alpha q}-M_{31} \lambda^{1-\alpha r}=: \xi(\lambda).
	\end{split}
    \end{align}
    Since $\alpha<\frac{1}{r-p}$ there holds $-\alpha p<1-\alpha r$ and recall $q<p<r$. Therefore,
    \begin{align}\label{sec6_13}
	\xi(\lambda) \to + \infty \quad \text{as } \lambda \to 0^+.
    \end{align}
    Hence, there exists a number $\lambda_1^*>0$ such that $\xi(\lambda)>0$ for all $\lambda \in \l(0,\lambda_1^*\r)$. Then, from \eqref{sec6_3} one has
    \begin{align}\label{sec6_10}
	\ph_\lambda(u) \geq \xi(\lambda)>0=\ph_\lambda(0)
    \end{align}
    for all $u \in \W1p0$ with $\|u\|_{\W1p0}=\lambda^{-\alpha}$ and $\lambda \in (0,\lambda_1^*)$.
    
    As before, thanks to hypotheses H(f)$_4$(i),(ii), we derive
    \begin{align}\label{sec6_11}
	\varphi_\lambda(t \hat{u}_1(p)) \to - \infty \quad \text{as } t \to +\infty \quad \text{for all }\lambda>0.
    \end{align}
    Finally, Proposition \ref{proposition_C_condition} ensures that $\ph_\lambda$ satisfies the $C$-condition. This fact along with \eqref{sec6_10} and \eqref{sec6_11} allow us to apply the mountain pass theorem stated in Theorem \ref{theorem_mountain_pass}. This yields an element $u_\lambda \in \W1p0$ such that
    \begin{align}\label{sec6_6}
	u_\lambda \in K_{\ph_\lambda}\setminus \{0\} \quad \text{and} \quad \xi(\lambda) \leq \ph_\lambda(u_\lambda).
    \end{align}
    Hence, $u_\lambda$ is a nontrivial solution of \eqref{problem3}. As before, the nonlinear regularity theory (see \cite{1968-Ladyzhenskaya-Ural'tseva}, \cite{1991-Lieberman})  and the nonlinear maximum principle (see \cite{2007-Pucci-Serrin} and hypothesis H(f)$_4$(v)) imply that $u_\lambda \in \interior$. Now, by applying \eqref{sec6_6}, Corollary \ref{corollary_upper_lower_estimates} and hypothesis H(f)$_4$(i), it follows
    \begin{align}\label{sec6_12}
	\xi(\lambda) \leq \ph_\lambda (u_\lambda)\leq M_{32} \l (1+\|u_\lambda\|_{\W1p0}^r\r)
    \end{align}
    for some $M_{32}>0$. The statement in \eqref{sec6_12} along with \eqref{sec6_13} yields that
    \begin{align*}
	\|u_\lambda\|_{\W1p0} \to \infty \quad \text{ as } \lambda \to 0^+.
    \end{align*}
    
    Now let us prove the second assertion of the theorem. To this end, recall that we have again
    \begin{align}\label{sec6_14}
	\begin{split}
	    \varphi_\lambda(u) \geq \frac{c_1}{p(p-1)} \|u\|_{\W1p0}^p- \lambda \left[M_{30}\|u\|^{q}_{\W1p0}+M_{31} \|u\|^r_{\W1p0} \right]
	\end{split}
    \end{align}
    for all $u \in \W1p0$ (see \eqref{sec6_2}). Let $\beta \in \left(0,\frac{1}{p} \right)$ and set $\|u\|_{\W1p0}=\lambda^{\beta}$. Then, \eqref{sec6_14} becomes
    \begin{align*}
	\begin{split}
	    \varphi_\lambda(u) \geq \frac{c_1}{p(p-1)} \lambda^{\beta p}-M_{30}\lambda^{1+\beta q}-M_{31} \lambda^{1+\beta r}=:\omega(\lambda).
	\end{split}
    \end{align*}
    Since
    \begin{align*}
	\omega(\lambda)=\lambda \left[ \frac{c_1}{p(p-1)} \lambda^{\beta p-1}- M_{24}\lambda^{\beta q}-M_{25} \lambda^{\beta r}\right]
    \end{align*}
    and $\beta p-1<0$, we see that
    \begin{align*}
	\omega(\lambda) \to +\infty \quad \text{as } \lambda \to 0^+.
    \end{align*}
    Therefore we find a number $\lambda_2^*>0$ such that
    \begin{align}\label{sec6_16}
	\ph_\lambda(u) \geq \omega(\lambda)>0=\ph_\lambda(0)
    \end{align}
    for all $u \in \W1p0$ with $\|u\|_{\W1p0}=\lambda^{\beta}$ and $\lambda \in \l(0,\lambda_2^*\r)$.
    
    Let $\overline{B}_\lambda=\l \{u \in \W1p0: \|u\|_{\W1p0} \leq \lambda^\beta \r\}$. By means of hypotheses H(a)$_2$(v) and H(f)$_4$(iv) we obtain, for $t \in (0,1)$ sufficiently small, that
    \begin{align*}
	\ph_\lambda \l(t \hat{u}_1(q) \r)<0
    \end{align*}
    (cf. the proof of Proposition \ref{proposition_Sec4_two_constant_sign}). Therefore
    \begin{align*}
	\inf_{\partial \overline{B}_\lambda}\ph_\lambda \geq \omega(\lambda)>0 \quad \text{and} \quad \inf_{\overline{B}_\lambda}\ph_\lambda<0.
    \end{align*}
    Set $d_\lambda:=\inf_{\partial \overline{B}_\lambda}\ph_\lambda-\inf_{\overline{B}_\lambda}\ph_\lambda$ and let $\eps \in \l(0,d_\lambda\r)$. Taking into account the Ekeland variational principle (see, for example, Gasi{\'n}ski-Papageorgiou \cite[p. 579]{2006-Gasinski-Papageorgiou}) there exists $u_\eps \in \overline{B}_\lambda$ such that
    \begin{align}\label{sec6_17}
	\ph_\lambda\l(u_\eps\r)\leq \inf_{\overline{B}_\lambda} \ph_\lambda+\eps
    \end{align}
    and
    \begin{align}\label{sec6_18}
	\ph_\lambda \l(u_\eps \r)\leq\ph_\lambda (y)+\eps \l\|y-u_\eps\r\|_{\W1p0} \quad \text{for all }y \in \overline{B}_\lambda.
    \end{align}
    Since $\eps <d_\lambda$, we infer from \eqref{sec6_17} that
    \begin{align*}
	\ph_\lambda\l(u_\eps\r)< \inf_{\partial \overline{B}_\lambda} \ph_\lambda,
    \end{align*}
    thus $u_\eps \in B_\lambda = \l\{u \in \W1p0: \|u\|_{\W1p0}<\lambda^\beta \r\}$. This ensures that $u_\eps+th \in \overline{B}_\lambda$ for every $h \in \W1p0$ and for all $t>0$ sufficiently small. Taking $y=u_\eps+th$ in \eqref{sec6_18} for $h \in \W1p0$ with such a small $t>0$, then dividing by $t>0$ and letting $t\to 0^+$, we obtain
    \begin{align*}
	-\eps \|h\|_{\W1p0} \leq \l\lan \ph'_\lambda\l(u_\eps\r),h\r\ran.
    \end{align*}
    Since $h \in \W1p0$ is arbitrary the last inequality gives $\l\|\ph'_\lambda\l(u_\eps\r)\r\|_*\leq \eps$.
    
    Now, let $\eps_n \to 0^+$ and let $u_n=u_{\eps_n}$. Hence,
    \begin{align*}
	\l(1+\l\|u_n\r\|_{\W1p0}\r) \ph'_\lambda \l(u_n\r) \to 0
    \end{align*}
    which in view of Proposition \ref{proposition_C_condition} implies that $u_n \to v_\lambda$ in $\W1p0$ for some $v_\lambda \in \W1p0$.
    
    Passing to the limit in \eqref{sec6_17} as $n \to \infty$ we have
    \begin{align*}
	\ph_\lambda \l(v_\lambda\r)=\inf_{\overline{B}_\lambda} \ph_\lambda<0=\ph_\lambda(0)
    \end{align*}
    which means that $v_\lambda \neq 0$ being a local minimizer of $\ph_\lambda$. Therefore, $v_\lambda$ is a solution of \eqref{problem3} and $v_\lambda \in \interior$ (as before).
    Moreover, since $u_\lambda$ is a critical point of $\ph_\lambda$ of mountain pass type, it follows that $v_\lambda \neq u_\lambda$. Finally, note that
    \begin{align*}
	\l\|v_\lambda\r\|_{\W1p0} < \lambda^ \beta.
    \end{align*}
    Thus, $\l\|v_\lambda\r\|_{\W1p0} \to 0$ as $\lambda\to 0^+$. Letting $\lambda^*=\min \l\{\lambda^*_1,\lambda^*_2\r\}$ we have the conclusion of our theorem.
\end{proof}

\end{document}